\documentclass[a4paper,11pt,reqno,noindent]{amsart}
\usepackage[centertags]{amsmath}
\usepackage{amsfonts,amssymb,amsthm,dsfont,cases,amscd,esint,enumerate}
\usepackage[T1]{fontenc}
\usepackage[english]{babel}
\usepackage[applemac]{inputenc}
\usepackage{newlfont}
\usepackage{color}
\usepackage[body={15cm,21.5cm},centering]{geometry} 
\usepackage{fancyhdr}
\usepackage{enumerate}

\theoremstyle{plain}
\newtheorem{theor10}{Theorem}
\newenvironment{theor1}
  {\pushQED{\qed}\begin{theor10}}
  {\popQED\end{theor10}}
\newtheorem{prop10}[theor10]{Proposition}
\newenvironment{prop1}
  {\pushQED{\qed}\begin{prop10}}
  {\popQED\end{prop10}}
\newtheorem{cor10}[theor10]{Corollary}

\newtheorem{lem10}[theor10]{Lemma}

\newtheorem{theor0}{Theorem}[section]
\newenvironment{theor}
  {\pushQED{\qed}\begin{theor0}}
  {\popQED\end{theor0}}
\newtheorem{lem0}[theor0]{Lemma}
\newenvironment{lem}
  {\pushQED{\qed}\begin{lem0}}
  {\popQED\end{lem0}}
\newtheorem{prop0}[theor0]{Proposition}

\newtheorem{cor0}[theor0]{Corollary}
\newenvironment{cor}
  {\pushQED{\qed}\begin{cor0}}
  {\popQED\end{cor0}}
\newtheorem{propr0}[theor0]{Property}

\newtheorem{hyp0}[theor0]{Hypothesis}

\newtheorem{result0}[theor0]{Result}

\newtheorem{conj0}[theor0]{Conjecture}

\newtheorem{heur0}[theor0]{Heuristics}

\newtheorem{defin0}[theor0]{Definition}
\newenvironment{defin}
  {\pushQED{\qed}\begin{defin0}}
  {\popQED\end{defin0}}
\theoremstyle{definition}
\newtheorem{rems0}[theor0]{Remarks}

\newtheorem{ex0}[theor0]{Example}

\newtheorem{exs0}[theor0]{Examples}

\newtheorem{rem0}[theor0]{Remark}
\newenvironment{rem}
  {\pushQED{\qed}\begin{rem0}}
  {\popQED\end{rem0}}
\newtheorem{qu0}[theor0]{Question}

\newtheorem{qus0}[theor0]{Questions}

\theoremstyle{plain}
\newtheorem{as0}[theor0]{Assumption}

\newtheorem*{asn0*}{\assumptionnumber}
  \providecommand{\assumptionnumber}{}
  \makeatletter
  \newenvironment{asn0}[2]
   {\renewcommand{\assumptionnumber}{Condition \!(#1) {\normalfont--- #2}}
    \begin{asn0*}
    \protected@edef\@currentlabel{{\normalfont(#1)}}}
   {\end{asn0*}}
  \makeatother
\newenvironment{asn}
  {\pushQED{\qed}\begin{asn0}}
  {\popQED\end{asn0}}

\mathchardef\emptyset="001F
\numberwithin{equation}{section}

\newcommand{\N}{\mathbb N}

\newcommand{\e}{\varepsilon}

\newcommand{\Pc}{\mathcal{P}}

\newcommand{\Xc}{\mathcal X}

\newcommand{\dist}{\operatorname{dist}}

\newcommand{\R}{\mathbb R}
\newcommand{\Z}{\mathbb Z}

\newcommand{\Ic}{\mathcal I}

\newcommand{\Rc}{{\mathcal R}}

\newcommand{\Md}{\mathbb M}

\newcommand{\Nc}{\mathcal N}
\newcommand{\cvf}{\rightharpoonup}

\newcommand{\loc}{{\operatorname{loc}}}
\newcommand{\Id}{\operatorname{Id}}

\newcommand{\E}{\mathbb{E}}

\newcommand{\D}{\operatorname{D}}

\newcommand{\Aa}{\boldsymbol a}

\newcommand{\Bb}{\bar{\boldsymbol B}}

\newcommand{\Ld}{\operatorname{L}}

\newcommand{\diam}{\operatorname{diam}}
\newcommand{\Div}{{\operatorname{div}}}

\newcommand{\Sym}{{\operatorname{sym}}}
\newcommand{\Skew}{{\operatorname{skew}}}

\newcommand{\step}[1]{\noindent \textit{Step} #1.}
\newcommand{\substep}[1]{\noindent \textit{Substep} #1.}
\newcommand{\Pm}{\mathbb{P}}
\newcommand{\pr}[1]{\mathbb{P}\left[ #1 \right]}
\newcommand{\prmm}[1]{\mathbb{P}\,[ #1 ]}
\newcommand{\prm}[1]{\mathbb{P}\big[ #1 \big]}

\newcommand{\expec}[1]{\mathbb{E}\left[ #1 \right]}
\newcommand{\expecm}[1]{\mathbb{E}\big[ #1 \big]}

\newcommand{\expecM}[1]{\mathbb{E}\bigg[ #1 \bigg]}

\usepackage[colorlinks,citecolor=black,urlcolor=black]{hyperref}

\title[Continuum percolation in stochastic homogenization]{Continuum percolation in stochastic homogenization and the effective viscosity problem}

\author[M. Duerinckx]{Mitia Duerinckx}
\address[Mitia Duerinckx]{
Universit\'e Libre de Bruxelles, D\'epartement de Math\'ematique, 1050~Brussels, Belgium
\& Universit\'e Paris-Saclay, CNRS, Laboratoire de Math\'ematiques d'Orsay, 91405~Orsay, France
\& University of California, Los Angeles, Department of Mathematics, CA~90095, USA}
\email{mitia.duerinckx@ulb.be}
\author[A. Gloria]{Antoine Gloria}
\address[Antoine Gloria]{Sorbonne Universit\'e, CNRS, Universit\'e de Paris, Laboratoire Jacques-Louis Lions, 75005~Paris, France \& Institut Universitaire de France \& Universit\'e Libre de Bruxelles, D\'epartement de Math\'ematique, 1050~Brussels, Belgium}
\email{gloria@ljll.math.upmc.fr}

\begin{document}
\selectlanguage{english}

\begin{abstract}
This contribution is concerned with the effective viscosity problem, that is, the homogenization of the steady Stokes system with a random array of rigid particles, for which the main difficulty is the treatment of close particles. Standard approaches in the literature have addressed this issue by making \emph{moment assumptions on interparticle distances}. Such assumptions however prevent clustering of particles, which is not compatible with physically-relevant particle distributions. In this contribution, we take a different perspective and consider \emph{moment bounds on the size of clusters of close particles}. On the one hand, assuming such bounds, we construct correctors and prove homogenization (using a variational formulation and $\Gamma$-convergence to avoid delicate pressure issues). On the other hand, based on subcritical percolation techniques, these bounds are shown to hold for various mixing particle distributions with nontrivial clustering. As a by-product of the analysis, we also obtain similar homogenization results for compressible and incompressible linear elasticity with unbounded random stiffness.
\end{abstract}

\maketitle
\setcounter{tocdepth}{1}
\tableofcontents
\vspace{-0.5cm}

\section{Introduction}
This contribution is concerned with the stochastic homogenization of degenerate elliptic systems using subcritical continuum percolation arguments. Whereas there is a large body of literature on homogenization and invariance principles for the random conductance model on the Bernoulli percolation cluster on $\Z^d$ (see e.g.~the survey~\cite{Biskup-11} and references therein), much less is known so far in the continuum setting. Our main motivation to exploit continuum percolation in stochastic homogenization is the so-called effective viscosity problem (and the rigorous justification~\cite{Hofer-GV-20,DG-21b} of Einstein's effective viscosity formula~\cite{Einstein-05}). More precisely, we consider a random array of small rigid particles in a Stokes fluid: the rigidity of suspended particles hinders the fluid flow and is expected to effectively increase the fluid viscosity on large scales. If particles are uniformly separated by a fixed positive minimal distance, the definition of an effective viscosity in the sense of homogenization follows from standard arguments~\cite{DG-21a}. This uniform separation assumption is however unrealistic from a physical perspective, e.g.~\cite{BG-72a,BG-72b}. Inspired by~\cite{JKO94}, the first author relaxed this into moment assumptions on the interparticle distance~\cite{D-20}. Although expected to be optimal in a general stationary ergodic setting, these weaker assumptions still prohibit particle contacts in 3D and are not realistic either. 
Whereas moment assumptions on the interparticle distance contain only local information in terms of the two-particle density, we know from the iid discrete setting~\cite{Biskup-11} that more global information
can be quite relevant in homogenization and random walks in random environments --- such as the fact that large clusters of close particles are unlikely in a strongly mixing setting. 
In this contribution, we use subcritical continuum percolation arguments to control the size of clusters of particles and we prove homogenization in that context. Our results are twofold:
\begin{enumerate}[$\bullet$]
\item On the probabilistic side, we establish new and nearly-optimal moment bounds on the size of clusters of close particles for a large class of strongly mixing models, cf.~Theorem~\ref{th:as-perc}. The approach combines the notion of action radius that we introduced in~\cite{DG20b}, together with a renormalization argument developed by Duminil-Copin, Raoufi, and~Tassion in~\cite{DCRT-20} for subcritical percolation. We recover previous results on the Voronoi percolation model~\cite{DCRT-19} and on the Poisson Boolean model~\cite{DCRT-20}, although in a slightly suboptimal form, and we further obtain new results on other standard models, including Matérn hardcore processes and the random parking process~\cite{Penrose-01}
(which, unlike the previous models, yield disjoint inclusions, as needed for the application to the effective viscosity problem).
\smallskip\item On the analytical side, we prove homogenization for the effective viscosity problem based on such subcritical percolation estimates on clusters of close particles, cf.~Theorem~\ref{th:homog}.
The main analytical difficulty is to get around the incompressibility constraint andthe poor control of the pressure field in the direct proximity of particle contact points.
The oscillating test function method, div-curl arguments, and other PDE approaches to homogenization all require controlling the pressure field everywhere and lead unavoidably to restrictions on interparticle distances as in~\cite{DG-21a,D-20}. Instead, we start from the variational formulation of the problem in the set of divergence-free fields and we appeal to $\Gamma$-convergence arguments~\mbox{\cite{DalMaso-93,Braides-06}}. The subtle step is then the construction of a divergence-free recovery sequence, for which we combine subcritical percolation estimates with Bogovskii's construction in John domains~\cite{ADM-06}.
This variational argument precisely allows to avoid describing details of the fluid velocity and pressure near particle contact points.
As a by-product, we also deduce similar homogenization results for compressible and incompressible linear elasticity  with stiff inclusions
(the argument is of course much simpler in the compressible case),
cf.~Theorem~\ref{th:deg-ell}.
\end{enumerate}
Before we turn to the precise description of our results, let us mention the recent related work~\cite{GV-GL-21} by G\'erard-Varet and Girodroux-Lavigne  in the setting of homogenization of
stiff inclusions in the {\it scalar} conductivity problem, for which 
they also relax conditions on interparticle distances (albeit in a different spirit).
Since they focus on the simpler scalar problem, they avoid the subtle incompressibility issues and may rely on standard PDE approaches to homogenization. This scalar problem is a particular case of the compressible elasticity problem covered in Section~\ref{sec:elasticity}.

\section{Main results}
\subsection{Subcritical percolation}\label{sec:perco-results}
We start with suitable assumptions on the random set of inclusions. Given an underlying probability space $(\Omega,\Pm)$, let $\Pc=\{x_n\}_n$ be a random point process on $\R^d$, consider a collection $\{I_n^\circ\}_n$ of random shapes, where each $I_n^\circ$ is a random connected bounded open set, and define the corresponding inclusions $I_n:=x_n+I_n^\circ$. Note that random shapes $\{I_n^\circ\}_n$ may depend on~$\Pc$ and that inclusions are (momentarily) allowed to overlap. We consider the random set $\Ic:=\bigcup_nI_n$, which we assume to satisfy the following.

\begin{asn}{H}{Stationarity and ergodicity}\label{Hd}$ $\\
The point process $\Pc$ and the inclusion process~$\Ic$ are stationary and ergodic.
\end{asn}

We need some control on the number of close inclusions, which we express as moment bounds on the size of clusters of ``too'' close inclusions. 
For that purpose, given a threshold parameter~$\rho>0$, we consider the family $\{K_{p,\rho}\}_{p\in \N}$ of connected components of the fattened set $\Ic+\rho B$ (with $B$ the unit ball), and we then define the corresponding clusters
\begin{equation}\label{eq:clusters}
J_{p,\rho}\,:=\,\bigcup_{I_n\subset K_{p,\rho}}I_n.
\end{equation}
For notational convenience, we choose the label $n=0$ such that $I_0$ is the inclusion that is the closest to the origin (or the first such inclusion in the lexigographic order if it is not unique), and we choose the label $p=0$ such that the cluster $J_{0,\rho}$ contains $I_0$.
In these terms, we shall require that the law of the diameter of clusters has some algebraic decay.

\begin{asn}{Clust$_{\rho,\kappa}$}{Cluster decay condition of order $\kappa>0$ at scale~$\rho>0$}\label{Hd'}$ $\\
There exists a constant $C_0$ such that
\begin{equation}\label{eq:decay-J0}
\pr{\diam J_{0,\rho}>r}\,\le\, C_0r^{-\kappa}\qquad\text{for all $r>0$}.\qedhere
\end{equation}
\end{asn}

Before stating homogenization results under this condition for suitable $\kappa$, we discuss the validity of the latter.
In view of subcritical percolation theory, the condition~\ref{Hd'} is expected to be a consequence of
the following smallness condition, for some large enough constant $L$, whenever the random set~$\Ic$ has strong enough mixing properties; this is made rigorous in Theorem~\ref{th:as-perc} below.

\begin{asn}{Perc$_L$}{Subcritical percolation condition with constant $L\ge 1$}\label{Hd''}$ $\\
One of the following two conditions is satisfied:
\begin{enumerate}[$\bullet$]
\item \emph{Smallness:}
There exist $h_0,r_0>0$ such that
\[r_0^{d-1}\,\pr{\diam J_{0,h_0}>r_0}\,\le\,\tfrac1Lh_0^d.\]
\item \emph{Asymptotic smallness:}
There exists $\delta>0$ such that random shapes $\{I_n^\circ\}_n$ almost surely satisfy an interior ball condition with radius $\delta$, and there exists $r_0>0$ such that
\begin{equation}\label{eq:smallness}
r_0^{d-1}\,\limsup_{h\downarrow0}\,\pr{\diam J_{0,h}>r_0}\,\le\,\tfrac1L\delta^d.
\qedhere
\end{equation}
\end{enumerate}
\end{asn}

\begin{rem}\label{rem:no-coll}
The asymptotic smallness condition~\eqref{eq:smallness} is easily checked to hold whenever there is almost surely no contact between inclusions. More precisely, under condition~\ref{Hd}, provided that we have in addition
\begin{enumerate}[$\bullet$]
\item \emph{Absence of contacts:} $\dist(I_m,I_n)>0$ almost surely for all $m\ne n$;
\smallskip\item \emph{Local finiteness:} $\sum_{j=1}^\infty\prm{\sharp\{n:I_n\cap B\ne\varnothing\}\ge j}^\alpha<\infty$ for some $\alpha<1$;
\item \emph{Control on inclusion diameters:} $r^{d-1}\pr{\diam I_0>r}\to0$ as~\mbox{$r\uparrow\infty$} (this is trivially satisfied if inclusion diameters are bounded by a deterministic constant);
\end{enumerate}
then the asymptotic smallness condition~\eqref{eq:smallness} follows. In fact, the following stronger version holds in that case,
\begin{equation}\label{eq:smallness+}
\limsup_{r\uparrow\infty}\,r^{d-1}\limsup_{h\downarrow0}\,\pr{\diam J_{0,h}>r}\,=\,0.
\end{equation}
A proof is given in Remark~\ref{rem:no-coll-pr}.
\end{rem}

Under this condition~\ref{Hd''}, if the random set $\Ic$ has a finite range of dependence, then~\ref{Hd'} can be deduced from standard Bernoulli percolation via stochastic domination~\cite{LSS-97}, in which case the decay rate in~\eqref{eq:decay-J0} is exponential. The same exponential decay is obtained in~\cite{DCRT-19} for the Voronoi percolation model.
The richer setting of the Poisson Boolean model has also been analyzed in~\cite{DCRT-20}, in which case the decay is seen to depend on the mixing rate via the distribution of random radii of the inclusions.
Beyond these examples, it remains an open question whether~\ref{Hd'} holds true under a general mixing condition.
In the spirit of our work~\cite{DG20b} on functional inequalities, we consider a large class of inclusion processes where mixing can be conveniently captured via the existence of so-called \emph{action radii}.
\begin{asn}{AR$_{\kappa}$}{Control of action radii of order $\kappa>0$}\label{HAR}$ $\\
The inclusion process $\Ic$ is measurable with respect to a collection $\Xc=(\Xc_y)_{y\in\Z^d}$ of iid random elements $\Xc_y$'s with values in some measurable space $M$.
Given an iid copy $\Xc'$ of~$\Xc$, there exists a so-called action radius for $\Ic$ with respect to $\Xc,\Xc'$ at any point $z\in\Z^d$, which is defined as a nonnegative $\sigma(\Xc,\Xc')$-measurable random variable $R^z$ such that almost surely
\[\qquad\Ic(\Xc^z)\cap\R^d\setminus B_{R^z}(z)\,=\,\Ic(\Xc)\cap\R^d\setminus B_{R^z}(z),\]
where the perturbed field $\Xc^z=(\Xc^z_y)_{y\in\Z^d}$ is defined by $\Xc^z_z=\Xc'_z$ and $\Xc^z_y=\Xc_y$ for all~$y\ne z$.
In addition, there exists $C_0>0$ such that
\[\qquad\pr{R^z>r}\,\le\, C_0r^{-(\kappa+d)}\qquad\text{for all $r>0$}.\qedhere\]
\end{asn}

The upcoming theorem states that Conditions~\ref{Hd''} and~\ref{HAR} imply~\ref{Hd'} for some~$\rho>0$ small enough; the proof is given in Section~\ref{sec:percol}. Note that our argument fails to reach exponential decay (cf.~restriction $\gamma<1$ below): this can be improved for the Voronoi percolation model and for the Poisson Boolean model by making a stronger use of the structure of the underlying Poisson process~\cite{DCRT-19,DCRT-20}, but it is not clear to us whether this is also possible in the more general setting of action radii.
\begin{theor1}[Cluster estimates in subcritical continuum percolation]\label{th:as-perc}
Assume that Conditions~\ref{Hd} and \ref{Hd''} hold for some $L$ large enough. Then there exists $\rho>0$ such that, if~\ref{HAR} holds for some~$\kappa>0$, then there exists $C_1>0$ such that
\[\pr{\diam J_{0,\rho}>r}\,\le\, C_1r^{-\kappa}\qquad\text{for all $r>0$}.\]
In addition, for any $\gamma<1$, if we further have $\pr{R^z>r}\le C_0\exp(-\frac1{C_0}r^\gamma)$, and if \ref{Hd''} holds for some $L$ large enough (depending on $\gamma$), then there exist~$\rho,C_\gamma>0$ such that $\pr{\diam J_{0,\rho}>r}\le C_\gamma \exp(-\frac1{C_\gamma}r^\gamma)$.
\qedhere
\end{theor1}

In view of the construction of action radii in~\cite{DG20b}, we mention a number of concrete examples where the above applies. We essentially recover previous results on the Voronoi percolation model~\cite{DCRT-19} and on the Poisson Boolean model~\cite{DCRT-20}, and we further treat several new examples such as Matérn hardcore processes and the random parking process, which involve more complicated transformations of an underlying Poisson process and are relevant for our applications to the effective viscosity problem.
\begin{enumerate}[(a)]
\item \emph{Poisson Boolean model:} Given a Poisson process $\Pc=\{x_n\}_n$ with intensity $\lambda$, and given an independent sequence~$\{r_n\}_n$ of iid positive random variables, we consider the spherical inclusions $I_n:=B_{r_n}(x_n)$. If the law of random radii satisfies $\pr{r_n \ge r} \le C_0r^{-(\kappa+d)}$ for some~$\kappa>0$, then, using the construction of action radii for this model in~\cite[Section~3.5]{DG20b}, Theorem~\ref{th:as-perc} entails that  there exist~$\rho,C_1>0$ and $\lambda_0>0$ small enough such that, provided $\lambda\le\lambda_0$,
\[\pr{\diam J_{0,\rho}>r}\,\le\, C_1r^{-\kappa}\qquad\text{for all $r>0$}.\]
In addition, for any $\gamma<1$, if we have $\pr{r_n \ge r} \le C_0\exp(-\frac1{C_0}r^\gamma)$, there exist $\rho,C_\gamma>0$ and $\lambda_\gamma>0$ small enough such that we have $\pr{\diam J_{0,\rho}>r}\le C_\gamma\exp(-\tfrac1{C_\gamma}r^\gamma)$ provided $\lambda\le\lambda_\gamma$.
The same was obtained in~\cite{DCRT-20}, where the critical case with exponential decay $\gamma=1$ is further covered.
\smallskip\item \emph{Voronoi percolation model:} Given a Poisson process $\Pc=\{x_n\}_n$, we consider the associated Voronoi tessellation $\mathcal V=\{V_n\}_n$ of $\R^d$ (the Delaunay tessellation would do as well).
Given an independent sequence $\{b_n\}_n$ of iid Bernoulli variables with $\pr{b_n=1}=q$, we define $\Ic:=\bigcup_{n:b_n=1}V_n$.
Using the construction of action radii for Poisson random tessellations in~\cite[Section~3.2]{DG20b}, Theorem~\ref{th:as-perc} entails that for any $\gamma<1$ there exist~$\rho,C_\gamma>0$ and $q_\gamma>0$ small enough such that for $q\le q_\gamma$,
\[\pr{\diam J_{0,\rho}>r}\,\le \,  C_\gamma\exp(-\tfrac1{C_\gamma}r^\gamma)  \qquad\text{for all $r>0$}.\]
The same was obtained in~\cite{DCRT-19} in the optimal form with $\gamma=1$.
\smallskip\item \emph{Another Voronoi percolation model:} Given a Poisson process $\Pc=\{x_n\}_n$ and the associated Voronoi tessellation $\mathcal V=\{V_n\}_n$ of $\R^d$ (the Delaunay tessellation would do as well), we denote by $r_n$ the diameter of the cell $V_n$, and for given constants $\lambda^+,\lambda^->0$ we consider the random sets $\Ic^+:=\bigcup_{n:r_n > \lambda^+} V_n$ and $\Ic^-:=\bigcup_{n:r_n < \lambda^-} V_n$.
As in~(b), Theorem~\ref{th:as-perc} entails that for any $\gamma<1$ there exist $\rho,C_\gamma>0$ and there exist $\lambda_\gamma^+$ large enough and $\lambda_\gamma^-$ small enough such that for $\lambda^+\ge\lambda^+_\gamma$ and $\lambda^-\le\lambda^-_\gamma$,
\[\prm{\diam J_{0,\rho}^+>r}+\prm{\diam J_{0,\rho}^->r}\,\le \,  C_\gamma\exp(-\tfrac1{C_\gamma}r^\gamma)  \qquad\text{for all $r>0$}.\]
We conjecture that the same holds with $\gamma=1$, which is no longer covered by~\cite{DCRT-19,DCRT-20}.
\smallskip\item \emph{Matérn hardcore processes:} Given a Poisson process $\Pc^\circ=\{x_n^\circ\}_n$, Matérn processes are hardcore processes obtained by applying different types of thinning rules to $\Pc^\circ$:
\begin{enumerate}[---]
\item The Matérn~I process is obtained by deleting every point of $\Pc^\circ$ that is at distance~\mbox{$<1$} of some other point.
\item Given an independent sequence $\{v_n\}_n$ of iid uniform random variables on $[0,1]$, the Matérn~II process is obtained by deleting every point $x_n^\circ\in\Pc^\circ$ that is at distance~\mbox{$<1$} of some other point $x_m^\circ$ with $v_m<v_n$.
\item The Matérn~III process is obtained by an iterative graphical construction, see~\cite{Penrose-01}: intuitively, we look at particles with increasing marks $v_n$'s and we now delete every point $x_n^\circ\in\Pc^\circ$ that is at distance $<1$ of some other point $x_m^\circ$ with $v_m<v_n$ {\it that has not yet been deleted}.
\end{enumerate}
Denoting by $\Pc=\{x_n\}_n$ one of these hardcore processes, we consider the disjoint inclusions $I_n:=B_{1/2}(x_n)$.
Using the construction of action radii for such processes in~\cite[Section~3.4]{DG20b}, Theorem~\ref{th:as-perc} entails that for any $\gamma<1$ there exist $\rho,C_\gamma>0$ such that
\[\pr{\diam J_{0,\rho}>r}\,\le \,  C_\gamma\exp(-\tfrac1{C_\gamma}r^\gamma)  \qquad\text{for all $r>0$},\]
and we conjecture that the same holds with $\gamma=1$. Many variations can be considered around Matérn's thinning procedures and we can further consider associated inclusion models with random radii or associated tessellations as above: the present result is easily adapted to all such processes.
\smallskip\item \emph{Random parking process:} The random parking process $\Pc=\{x_n\}_n$ is obtained as the jamming limit of the Matérn III process, see~\cite{Penrose-01}, and we consider the corresponding disjoint inclusions $I_n:=B_{1/2}(x_n)$.
Using the construction of action radii in~\cite[Section~3.3]{DG20b}, Theorem~\ref{th:as-perc} entails that for any $\gamma<1$ there exist $\rho,C_\gamma>0$ such that
\[\pr{\diam J_{0,\rho}>r}\,\le \,  C_\gamma\exp(-\tfrac1{C_\gamma}r^\gamma)  \qquad\text{for all $r>0$},\]
and we again conjecture that the same holds with $\gamma=1$.
\end{enumerate}

\subsection{Effective viscosity problem}
We turn to the study of the effective viscosity problem based on the above percolation results of Section~\ref{sec:perco-results}. Denote by $d\ge 2$ the spatial dimension.
Since inclusions now represent rigid particles suspended in a fluid, we further assume that they are disjoint:

\begin{asn}{Hard}{Hardcore condition}\label{Hh}$ $\\
For all $n\ne m$ we have
$I_n\cap I_m=\varnothing$ almost surely,
that is, inclusions may touch but may not overlap.
\qedhere
\end{asn}

We consider the homogenization problem for the Stokes system in a bounded Lipschitz domain $U\subset\R^d$ in presence of a stationary random suspension of small rigid particles. More precisely, given a random inclusion process $\Ic=\bigcup_nI_n$ satisfying Assumptions~\ref{Hd},~\ref{Hh}, and~\ref{Hd'} for some $\rho,\kappa>0$, we denote by
\[\Ic_\e(U)\,:=\,\bigcup_{n\in\Nc_\e(U)}\e I_n\]
the set of small rigid particles in the domain $U$, with characteristic size $\e>0$, where for some boundary layer $\theta_\e>0$ we let $\Nc_\e(U)\subset\N$ be some random subset of indices such that
\[\big\{n\,:\,\dist(\e I_n,\partial U)\ge\theta_\e\big\}~~\subset~~\Nc_\e(U)~~\subset~~\big\{n\,:\,\e I_n\subset U\big\},\]
with $\theta_\e\to0$ as $\e\downarrow0$.
For such a set~$\Ic_\e(U)$ of small rigid particles in $U$, given an internal force~$f\in\Ld^2(U)^d$, the associated Stokes problem takes form of the following variational problem\footnote{This minimization problem is clearly well-posed: due to the rigidity constraint the functional is equivalent to~$\int_U 2|\!\D(v)|^2 - f\cdot v\,\mathds1_{U\setminus\Ic_\e(U)}$, the latter is weakly lower semicontinuous and coercive on $H^1_0(U)^d$ by Korn's inequality, and the incompressibility and rigidity constraints are feasible and weakly closed.},
\begin{equation}\label{eq:mini-Stokes}
\inf\bigg\{\int_{U\setminus\Ic_\e(U)}|\!\D(v)|^2-f \cdot v~:~v \in H^1_0(U)^d,~\Div(v)=0,~\D(v)|_{\Ic_\e(U)}=0\bigg\},
\end{equation}
where $\D(v)$ stands for the symmetrized gradient of $v$.
We denote by $u_\e$ the unique minimizer, which represents the fluid velocity field.
In the absence of contacts between particles in the sense of Remark~\ref{rem:no-coll}, we note that the Euler--Lagrange equations for $u_\e$ take the following standard form,
\begin{equation*}
\left\{\begin{array}{ll}
-\triangle u_\e+\nabla P_\e=f,&\text{in $U\setminus\Ic_\e(U)$},\\
\Div(u_\e)=0,&\text{in $U\setminus\Ic_\e(U)$},\\
\D(u_\e)=0,&\text{in $\Ic_\e(U)$},\\
u_\e=0,&\text{on $\partial U$},\\
\int_{\e I_n}\sigma(u_\e,P_\e)\nu=0,&\text{for all $n\in\Nc_\e(U)$},\\
\int_{\e I_n}\Theta(x-\e x_n)\cdot\sigma(u_\e,P_\e)\nu=0,&\text{for all $n\in\Nc_\e(U)$ and $\Theta\in\Md^\Skew$},
\end{array}\right.
\end{equation*}
where $\sigma(u_\e,P_\e):=2\D(u_\e)-P_\e\Id$ is the Cauchy stress tensor, where $\nu$ denotes the outward unit normal vector at particle boundaries, and where $\Md^\Skew\subset\R^{d\times d}$ stands for the set of skew-symmetric matrices.

\medskip
Before addressing the homogenization limit $\e\downarrow0$ of this Stokes problem~\eqref{eq:mini-Stokes}, we start with the relevant definition of associated correctors and effective viscosity. This generalizes the results of~\cite{DG-21a,D-20} without uniform particle separation in the subcritical percolation framework of~\ref{Hd'};  the proof, which follows from~\cite{DG-21a} together with a simple approximation argument as in~\cite{D-20},  is postponed to Section~\ref{sec:cor}. 
The definition~\eqref{eq:def-B} of the effective viscosity $\Bb$ is the starting point for the justification of  Einstein's effective viscosity formula in the dilute limit in~\cite{DG-21b} in the same percolation context.

\begin{prop1}[Correctors for the effective viscosity problem]\label{prop:cor}
Assume that the inclusion process $\Ic$ satisfies~\ref{Hd}, \ref{Hh}, and~\ref{Hd'} for some $\rho>0$ and some $\kappa$ large enough. For all $E\in\Md_0^\Sym$ (the set of trace-free symmetric matrices), there exists a unique minimizer~$\D(\psi_E)$ of the variational problem
\begin{multline}\label{eq:cor-variat}
\inf\Big\{\expec{|\!\D(\psi)+E|^2}~:~\psi\in\Ld^2(\Omega;H^1_\loc(\R^d)^d),~\nabla\psi~\text{stationary},\\
\Div(\psi)=0,~(\D(\psi)+E)|_{\Ic}=0,~\expec{\nabla\psi}=0\Big\},
\end{multline}
and the minimum value defines a positive-definite symmetric linear map $\Bb$ on $\Md^\Sym_0$, that is, the so-called \emph{effective viscosity},
\begin{equation}\label{eq:def-B}
E:\Bb E\,:=\,\expec{|\!\D(\psi_E)+E|^2}.
\end{equation}
Given $\D(\psi_E)$, the field $\psi_E\in\Ld^2(\Omega;H^1_\loc(\R^d)^d)$ is itself uniquely defined with anchoring $\int_B\psi_E=0$ and $\int_B\nabla\psi_E\in\Md_0^\Sym$.
In particular, $\psi_E$ is sublinear at infinity in the following sense: almost surely, as $\e\downarrow0$,
\begin{equation}\label{eq:conv-cor}
\begin{array}{rll}
(\nabla\psi_E)(\tfrac\cdot\e)~\cvf~0&\text{weakly}&\text{in $\Ld^2_\loc(\R^d)^{d\times d}$},\\
\e\psi_E(\tfrac\cdot\e)~\to~0&\text{strongly}&\text{in $\Ld^q_\loc(\R^d)^d$,~~~~~~~for all $q<\frac{2d}{d-2}$}.
\end{array}
\qedhere
\end{equation}
\end{prop1}

With these objects at hand, we may now state our homogenization result for the Stokes problem~\eqref{eq:mini-Stokes}; the proof is postponed to Section~\ref{sec:Gamma}.
This generalizes the results of~\cite{DG-21a,D-20} without uniform particle separation in the subcritical percolation framework of~\ref{Hd'}.
Previous results on this homogenization problem~\cite{DG-21a,D-20} were based on the oscillating test function method or on the div-curl lemma, but those PDE approaches cannot be adapted to the present setting due to the poor control of the pressure near particle contact points. Instead, we appeal to a novel $\Gamma$-convergence argument, which allows to completely bypass pressure degeneracy issues.

\begin{theor1}[Homogenization for effective viscosity problem]\label{th:homog}
Let the assumptions of Proposition~\ref{prop:cor} hold for some $\kappa$ large enough.
Given $f\in\Ld^2(U)^d$, the sequence $(u_\e)_\e$ of minimizers of the Stokes problem~\eqref{eq:mini-Stokes} almost surely converges weakly in $H^1_0(U)^d$ as~$\e\downarrow0$ to the unique minimizer~$\bar u$ of the effective variational problem
\begin{equation}\label{eq:mini-Stokes-hom}
\inf\bigg\{\int_U \D(v): \Bb \D(v)-(1-\lambda) f \cdot v~:~v \in H^1_0(U)^d,~\Div(v)=0\bigg\},
\end{equation}
in terms of the effective viscosity $\Bb$ defined in~\eqref{eq:def-B} and of the particle volume fraction $\lambda:=\expec{\mathds1_{\Ic}}$.
Equivalently, the limit $\bar u$ is the unique solution of the effective Stokes system
\begin{equation}\label{eq:Stokes-hom}
\left\{\begin{array}{ll}
-\nabla\cdot2\Bb\D(\bar u)+\nabla\bar P=(1-\lambda)f,&\text{in $U$},\\
\Div(\bar u)=0,&\text{in $U$},\\
\bar u=0,&\text{on $\partial U$}.
\end{array}\right.\qedhere
\end{equation}
\end{theor1}

\subsection{Homogenization of linear elasticity with stiff inclusions}\label{sec:elasticity}
We turn to the homogenization problem for compressible or incompressible linear elasticity with unbounded stiffness in dimensions $d\ge2$. Given a random symmetric $4$-tensor coefficient field $\Aa$ on~$\R^d$, representing the stiffness of the material, given a bounded Lipschitz domain~$U\subset\R^d$, and given an internal force $f\in\Ld^2(U)^d$, we consider the random families~$(u_\e^1)_\e$ and $(u_\e^2)_\e$ of minimizers of the associated energy functionals for compressible and incompressible linear elasticity, respectively,
\begin{eqnarray}
\hspace{-0.6cm}&&\inf\bigg\{\int_{U}\tfrac12 \D(v^1):\Aa (\tfrac\cdot\e)\D(v^1)-f \cdot v^1~:~v^1 \in H^1_0(U)^d\bigg\},\label{eq:ellPDE}
\\
\hspace{-0.6cm}&\text{and}&\inf\bigg\{\int_{U}\tfrac12 \D(v^2):\Aa (\tfrac\cdot\e)\D(v^2)-f \cdot v^2~:~v^2 \in H^1_0(U)^d,~\Div(v^2)=0\bigg\}.\label{eq:ellPDE2}
\end{eqnarray}
We assume that the coefficient field $\Aa$ is uniformly elliptic, but that it may be unbounded, corresponding to stiff zones. The case of a stiffness tensor $\Aa$ that is not uniformly elliptic could be considered as well, thus allowing for soft zones, but such zones are in fact easier to deal with, and we refer the reader to~\cite[Section~8.1]{JKO94}.

\begin{asn}{H$'$}{General conditions}\label{Kd}$ $
\begin{enumerate}[$\bullet$]
\item \emph{Stationarity and ergodicity:}\\
The random coefficient field $\Aa:\R^d\times\Omega\to\R^{d\times d\times d\times d}$ is stationary and ergodic.
\smallskip\item \emph{Uniform ellipticity:}\\
For all $E \in \Md^\Sym$ (or  $E \in \Md^\Sym_0$ in the incompressible case)
and all $x\in \R^d$, we have almost surely
\[E: \Aa(x) E  \ge |E|^2.\qedhere\]
\end{enumerate}
\end{asn}

As $\Aa$ is unbounded, we need some geometric control on regions where $\Aa$ is ``too'' large, and we shall exploit the same percolation arguments as in the previous section.
Given a threshold parameter~$\lambda>0$, we consider the connected components $\{I_{n,\lambda}\}_n$ of the random stationary set $\{x\in\R^d:|\Aa(x)|>\lambda\}$, and we define $\Ic_\lambda:=\bigcup_nI_{n,\lambda}$.
These inclusions may however be very close to one another, which we further need to control.
For that purpose, given another threshold parameter $\rho>0$, we consider the family $\{K_{p,\lambda,\rho}\}_p$ of connected components of the fattened set $\Ic_\lambda+\rho B$, and we then define the corresponding clusters
\[J_{p,\lambda,\rho}\,:=\,\bigcup_{I_{n,\lambda}\subset K_{p,\lambda,\rho}}I_{n,\lambda}.\]
For convenience, we choose the label $n=0$ such that $I_{0,\lambda}$ is the inclusion that is the closest to the origin, and we choose $p=0$ such that the cluster $J_{0,\lambda,\rho}$ contains $I_{0,\lambda}$.
In these terms, we shall require that the law of the diameter of clusters has some algebraic decay.

\begin{asn}{Clust$'_{\!\lambda,\rho,\kappa}$}{Cluster decay condition of order $\kappa>0$ at scales $\lambda,\rho>0$}\label{Kd'}$ $\\
There exists a constant $C_0$ such that
\[\pr{\diam J_{0,\lambda,\rho}>r}\,\le\, C_0\,r^{-\kappa}\qquad\text{for all $r>0$.}\qedhere\]
\end{asn}

Before addressing the homogenization limit $\e\downarrow0$ of~\eqref{eq:ellPDE} and~\eqref{eq:ellPDE2}, we start with the relevant definition of associated correctors and effective stiffness; the proof is postponed to Section~\ref{sec:stiff}.
Note that the proof is simpler and the assumptions are milder in the compressible case (we only require~$\kappa>2$ instead of $\kappa$ large enough in that case).

\begin{prop1}[Correctors for linear elasticity with unbounded stiffness]\label{prop:cor-ell}
Assume that the random coefficient field $\Aa$ satisfies~\ref{Kd} and~\ref{Kd'} for some $\lambda,\rho>0$, for some~$\kappa>2$ in the compressible case and for some~$\kappa$ large enough in the incompressible case.
\begin{enumerate}[(i)]
\item \emph{Compressible case:}\\
For all $E\in\Md^\Sym$ (the set of symmetric matrices), there exists a unique minimizer~$\D(\varphi_E^1)$ of the variational problem
\begin{multline}\label{eq:corr-phi1}
\qquad\inf\Big\{\expec{(\D(\phi)+E):\Aa(\D(\phi)+E)}~:~\phi\in\Ld^2(\Omega;H^1_\loc(\R^d)^d),\\
~\nabla\phi~\text{stationary},
~\expec{\nabla \phi}=0\Big\},
\end{multline}
and the minimum value defines a symmetric positive-definite $4$-tensor $\bar\Aa^1$ on $\Md^\Sym$, that is, the so-called compressible \emph{effective stiffness tensor},
\[E:\bar\Aa^1 E\,:=\,\expec{(\D(\varphi_E^1)+E):\Aa(\D(\varphi_E^1)+E)}.\]
\item \emph{Incompressible case:}\\
For all $E\in\Md^\Sym_0$, there exists a unique minimizer~$\D(\varphi_E^2)$ of the variational problem~\eqref{eq:corr-phi1} where we add the incompressibility constraint $\Div(\phi)=0$. The minimum value defines another symmetric positive-definite $4$-tensor $\bar\Aa^2$ on $\Md^\Sym_0$, that is, the so-called incompressible \emph{effective stiffness tensor},
\[E:\bar\Aa^2 E\,:=\,\expec{(\D(\varphi_E^2)+E):\Aa(\D(\varphi_E^2)+E)}.\]
\end{enumerate}
In both cases, the fields $\varphi_E^1,\varphi_E^2\in\Ld^2(\Omega;H^1_\loc(\R^d)^d)$ are themselves uniquely defined with anchoring $\int_B\varphi_E^1=\int_B\varphi_E^2=0$ and $\int_B\nabla\varphi_E^1,\int_B\nabla\varphi_E^2\in\Md^\Sym$, and they are sublinear at infinity as $\psi_E$ in~\eqref{eq:conv-cor}.
\end{prop1}

With these objects at hand, we may now state the homogenization result for compressible and incompressible linear elasticity problems~\eqref{eq:ellPDE} and~\eqref{eq:ellPDE2}
(which extends the classical results of \cite[Chapter~12]{JKO94}); the proof is postponed to Section~\ref{sec:stiff}. Again, the proof is drastically simpler and the assumptions are milder in the compressible case (we only require~$\kappa>d$ instead of $\kappa$ large enough in that case).

\begin{theor1}[Homogenization for linear elasticity with unbounded stiffness]\label{th:deg-ell}
Let the assumptions of Proposition~\ref{prop:cor-ell} hold for some $\kappa>d$ in the compressible case and for some~$\kappa$ large enough in the incompressible case.
Given $f\in\Ld^2(U)^d$, the sequences $(u_\e^1)_\e$ and $(u_\e^2)_\e$ of solutions of the compressible and incompressible linear elasticity problems~\eqref{eq:ellPDE} and~\eqref{eq:ellPDE2} converge weakly in $H^1_0(U)^d$ as $\e\downarrow0$ to the unique solutions $\bar u^1$ and $\bar u^2$ of the effective problems
\[\left\{\begin{array}{ll}
-\nabla\cdot\bar\Aa^1\D(\bar u^1)=f,&\text{in $U$},\\
\bar u^1=0,&\text{on $\partial U$},
\end{array}\right.
\qquad
\left\{\begin{array}{ll}
-\nabla\cdot\bar\Aa^2\D(\bar u^2)+\nabla\bar P^2=f,&\text{in $U$},\\
\Div(\bar u^2)=0,&\text{in $U$},\\
\bar u^2=0,&\text{on $\partial U$},
\end{array}\right.\]
in terms of the effective stiffness tensors $\bar\Aa^1$ and $\bar\Aa^2$ defined in Proposition~\ref{prop:cor-ell}.
\end{theor1}

\subsection*{Notation}
\begin{enumerate}[\quad$\bullet$]
\item For vector fields $u,u'$ and matrix fields $T,T'$, we set $(\nabla u)_{ij}=\nabla_ju_i$, $(\D(u))_{ij}=\frac12(\nabla_ju_i+\nabla_iu_j)$, $(\Div(T))_i=\nabla_jT_{ij}$, $T:T'=T_{ij}T'_{ij}$, $(u\otimes u')_{ij}=u_iu'_j$, where we systematically use Einstein's summation convention on repeated indices. For a matrix~$E$, we write {$\partial_Eu=E:\nabla u$}.
\smallskip\item We denote by $\Md^\Sym\subset\R^{d\times d}$ the set of symmetric matrices, by $\Md_0^\Sym$ the set of symmetric trace-free matrices, and by $\Md^\Skew$ the set of skew-symmetric matrices.
\smallskip\item The ball centered at $x$ of radius $r$ in $\R^d$ is denoted by $B_r(x)$, and we simply write $B(x)=B_1(x)$, $B_r=B_r(0)$, and $B=B_1(0)$. The unit cube centered at $z$ is similarly denoted by $Q(z)$.
\smallskip\item We denote by $C>0$ any constant (or exponent) that only depends on the dimension~$d$ and on the parameters appearing in the different assumptions. The value of $C$ is allowed to change from one line to another.
We use the notation $\lesssim$ (resp.~$\gtrsim$) for $\le C\times$ (resp.~$\ge\frac1C\times$) up to such a multiplicative constant~$C$. We add subscripts to $C,\lesssim,\gtrsim$ to indicate dependence on other parameters.
\end{enumerate}

\section{Cluster estimates in subcritical continuum percolation}\label{sec:percol}

This section is dedicated to the proof of Theorem~\ref{th:as-perc}.
We start with a general continuum percolation setting and postpone the proof of Theorem~\ref{th:as-perc} to Section~\ref{sec:proof-mainpercol}.
Let~$\eta$ be a $\Z^d$-stationary random field on $\R^d$ with $\eta|_{Q(z)}\in\{0,1\}$ almost surely for all $z\in\Z^d$.
We are interested in the size of connected components of $\eta^{-1}(\{1\})$ in the regime of subcritical percolation.
For $A,A'\subset\R^d$, we consider the event $\{A\leftrightarrow A'\}$ that there exists a chain of adjacent cubes with value $\eta=1$ such that the chain starts at a cube intersecting~$A$ and ends up at a cube intersecting~$A'$, and we consider the connectivity probability $\pr{\{0\}\leftrightarrow \partial B_r}$.
If~$\eta$ is a Bernoulli process, then the latter is known to decay exponentially in the distance~$r$, e.g.~\cite[Chapter~5]{Grimmett}.
If $\eta$ has a finite range of dependence, then stochastic domination applies and yields the same conclusion~\cite{LSS-97}.
Nevertheless, it remains an open question whether the same holds if $\eta$ only satisfies a strong mixing condition.
For the Voronoi percolation model and the Poisson Boolean model, this was recently settled  in~\cite{DCRT-19,DCRT-20}. We consider here a more general setting of correlated fields for which the decay of connectivity probabilities can be estimated nearly-optimally. Starting point is the following notion of {\it dependence radii}.

\begin{defin}\label{def:dep-rad}
For $z\in\Z^d$ and $r>0$, a \emph{dependence radius} for $\eta$ on $B_r(z)$, if it exists, is defined as a random variable $\Rc_{r}^z$ such that for all $\ell>0$ the restrictions $\eta|_{B_{r}(z)}$ and $\eta|_{\R^d\setminus B_{r+\ell}(z)}$ are independent when conditioned on the event that $\Rc_{r}^z\le\ell$.
\end{defin}

In order to construct dependence radii for a large class of examples of correlated fields, we relate it to the following notion of {\it action radii} that we introduced in~\cite{DG20b}.

\begin{defin}\label{eq:actionrad}
Let $\eta$ be measurable with respect to a collection $\Xc=(\Xc_y)_{y\in\Z^d}$ of iid random elements $\Xc_y$'s with values in a measurable space $M$.
Given an iid copy $\Xc'$ of~$\Xc$, an \emph{action radius} for $\eta$ with respect to $\Xc,\Xc'$ at any point $z\in\Z^d$, if it exists, is defined as a nonnegative $\sigma(\Xc,\Xc')$-measurable random variable $R^{z}$ such that almost surely,
\[\eta(\Xc^{z})|_{\R^d\setminus B_{R^z}(z)}\,=\,\eta(\Xc)|_{\R^d\setminus B_{R^z}(z)},\]
where the perturbed collection $\Xc^{z}=(\Xc^z_y)_{y\in\Z^d}$ is given by $\Xc^{z}_z:=\Xc'_z$ and $\Xc^{z}_y:=\Xc_y$ for all~$y\ne z$.
\end{defin}

The following lemma draws the link between dependence radii and action radii.
As a consequence, using our construction of action radii in~\cite{DG20b}, this allows us to construct and estimate dependence radii for various models of interest, including Poisson Boolean models, Voronoi percolation models, Matérn hardcore processes, and the random parking process.

\begin{lem}\label{lem:dep/act-rad}
Let $\eta$ be measurable with respect to a collection $\Xc=(\Xc_y)_{y\in\Z^d}$ of iid random elements $\Xc_y$'s with values in a measurable space $M$.
If there exists an action radius $R^{z}$ for~$\eta$ at any point~$z\in\Z^d$,
then $\eta$ admits a dependence radius on $B_r(z)$ for any $z,r$, which is given by
\begin{equation}\label{eq:depend-rad}
\Rc_r^z\,:=\,\inf\Big\{\ell>0\,:\,R^{y}\le(|z-y|-r)\vee(r+\ell-|z-y|),~\forall y\in\Z^d\Big\}.
\end{equation}
In particular, the following implications hold: given $C_0,\kappa>0$,
\begin{enumerate}[(i)]
\item If $\prmm{R^{z}>\ell}\le C_0\,\ell^{-(\kappa+d)}$ for all $\ell$,\\
then $\exists C_1$ such that $\prmm{\Rc_r^{z}>\ell}\le C_1(1+\tfrac r\ell)^{d-1}\ell^{-\kappa}$ \text{for all $r,\ell$}.
\smallskip
\item If $\prmm{R^{z}>\ell}\le C_0\exp(-\tfrac1{C_0}\ell^{\kappa})$ for all $\ell$,\\
then $\exists C_1$ such that $\prmm{\Rc_r^{z}>\ell}\le C_1(1+\tfrac r\ell)^{d-1}\exp(-\tfrac1{C_1}\ell^{\kappa})$ for all $r,\ell$.\qedhere
\end{enumerate}
\end{lem}

Next, based on this notion of dependence radii, we establish the following estimate on connectivity probabilities in the subcritical regime. As opposed to the (simpler) case of the Voronoi percolation model and of the Poisson Boolean model treated in~\cite{DCRT-19,DCRT-20},
we emphasize that the present argument does not allow  to prove exponential decay (see restriction~$\kappa<\kappa(\alpha)$ in item~(ii) where $\kappa(\alpha)\uparrow1$ as $\alpha\downarrow0$).

\begin{theor}\label{th:percol}
Assume that the random field $\eta$ admits a dependence radius in the sense of Definition~\ref{def:dep-rad}.
For any $\alpha<\frac12$, there exists a universal constant $\e_\alpha>0$ such that,
provided we have
\[\pr{B_{r_0/2}\leftrightarrow \partial B_{r_0}}<\e_\alpha\qquad\text{for some $r_0>0$},\]
the following implications hold: given $C_0,\kappa>0$,
\begin{enumerate}[(i)]
\item If $\pr{\Rc_r^z>2\alpha r}\le C_0r^{-\kappa}$ for all $r$,\\
then $\exists C_1$ such that $\pr{\{0\}\leftrightarrow\partial B_r}\le C_1r^{-\kappa}$ for all $r$.
\smallskip
\item If $\pr{\Rc_r^z>2\alpha r}\le C_0 \exp(-\frac1{_0}r^{\kappa})$ for all $r$, with $\kappa<\kappa(\alpha):=\frac{\log2}{\log 2-\log(1-2\alpha)}$,\\
then $\exists C_1$ such that $\pr{\{0\}\leftrightarrow\partial B_r}\le C_1\exp(-\frac1{C_1}r^{\kappa})$ for all $r$.
\qedhere
\end{enumerate}
\end{theor}

\subsection{Proof of Lemma~\ref{lem:dep/act-rad}}
We split the proof into two steps, first showing that $\Rc_r^z$ in~\eqref{eq:depend-rad} defines a dependence radius and then checking the bounds~(i) and~(ii).

\medskip
\step1 Proof that $ \Rc^z_r$ is a dependence radius.\\
By definition of action radii, cf.~Definition~\ref{eq:actionrad}, for all $z\in\Z^d$ and $\ell>0$, the condition~\mbox{$R^z\le \ell$} entails that $\eta|_{\R^d\setminus B_\ell(z)}$ is a function of $(\Xc_y)_{y:y\ne z}$.
Given $\Rc_r^z$ defined in~\eqref{eq:depend-rad}, conditioning on the event that $\Rc_r^z\le2\ell$, we get $R^{y}\le|z-y|-r$ for all $y\in\Z^d\setminus B_{r+\ell}(z)$ and $R^{y}\le r+\ell-|z-y|$ for all $y\in \Z^d\cap B_{r+\ell}(z)$, which entails that $\eta|_{B_r(z)}$ is a function of $(\Xc_y)_{y\in \Z^d\cap B_{r+\ell}(z)}$ and that $\eta|_{\R^d\setminus B_{r+2\ell}(z)}$ is a function of $(\Xc_y)_{y\in\Z^d\setminus B_{r+\ell}(z)}$. As $(\Xc_y)_{y}$ is an iid sequence, this implies that $\eta|_{B_r(z)}$ and $\eta|_{\R^d\setminus B_{r+2\ell}(z)}$ are independent, thus showing that $\Rc_r^z$ is a dependence radius for~$\eta$ on~$B_r(z)$ in the sense of Definition~\ref{def:dep-rad}.

\medskip
\step2 Union bound.\\
Starting from~\eqref{eq:depend-rad}, a union bound yields
\begin{multline*}
\pr{\Rc^z_r>2\ell}\,\le\,\sum_{y\in\Z^d} \prm{R^y>(|z-y|-r)\vee(r+2\ell-|z-y|)}\\
\,\le\,\sum_{y\in\Z^d\setminus B_{r+\ell}(z)} \prm{R^{y}>|z-y|-r}+\sum_{y\in \Z^d\cap B_{r+\ell}(z)} \prm{R^{y}>r+2\ell-|z-y|},
\end{multline*}
and the bounds~(i) and~(ii) follow from a direct calculation.\qed

\subsection{Proof of Theorem~\ref{th:percol}}
The proof is based on a buckling argument starting from a renormalization inequality as inspired by a recent work of Duminil-Copin, Raoufi, and Tassion~\cite{DCRT-20}.
For all $r>0$ and~$\alpha\in(0,1)$, we consider connectivity probabilities
\[\theta_r^\alpha\,:=\,\pr{B_{\alpha r}\leftrightarrow\partial B_r}.\]

\begin{lem}[Renormalization inequality]\label{lem:renorm}
For all $0<\alpha<\beta<1$ and $0<\ell<(1-\beta)r$, assuming that $\pr{\Rc_{\beta r}>\ell}\le\frac12$, we have
\[\theta_r^\alpha~\le~\pr{\Rc_{\beta r}>\ell}+C(\tfrac{\beta r+\ell}{r-\beta r-\ell}\tfrac{1}{\alpha(\beta-\alpha)})^{d-1}\,\theta^\alpha_{r(1-\beta)-\ell}\,
\theta^\alpha_{r(\beta-\alpha)}.
\qedhere\]
\end{lem}

\begin{proof}
Conditioning with respect to the event that $\Rc_{\beta r}\le\ell$, and using the property of the dependence radius, see Definition~\ref{def:dep-rad}, we find
\begin{eqnarray*}
\theta_r^\alpha&\le&\pr{\Rc_{\beta r}>\ell}+\pr{B_{\alpha r}\leftrightarrow\partial B_r,~\text{and}~\Rc_{\beta r}\le\ell}\\
&\le&\pr{\Rc_{\beta r}>\ell}+\pr{B_{\alpha r}\leftrightarrow\partial B_{\beta r},~B_{\beta r+\ell}\leftrightarrow\partial B_r,~\text{and}~\Rc_{\beta r}\le\ell}\\
&=&\pr{\Rc_{\beta r}>\ell}+\pr{B_{\alpha r}\leftrightarrow\partial B_{\beta r},~B_{\beta r+\ell}\leftrightarrow\partial B_r \,\|\,\Rc_{\beta r}\le\ell} \pr{\Rc_{\beta r}\le\ell}
\\
&=&\pr{\Rc_{\beta r}>\ell}+\pr{B_{\alpha r}\leftrightarrow\partial B_{\beta r} \,\|\,\Rc_{\beta r}\le\ell}\pr{B_{\beta r+\ell}\leftrightarrow\partial B_r \,\|\,\Rc_{\beta r}\le\ell} \pr{\Rc_{\beta r}\le\ell}
\\
&=&\pr{\Rc_{\beta r}>\ell}+\frac{\pr{B_{\alpha r}\leftrightarrow\partial B_{\beta r},~\text{and}~\Rc_{\beta r}\le\ell}\pr{B_{\beta r+\ell}\leftrightarrow\partial B_r,~\text{and}~\Rc_{\beta r}\le\ell}}{\pr{\Rc_{\beta r}\le\ell}}.
\end{eqnarray*}
Since $\pr{\Rc_{\beta r}>\ell}\le\frac12$, this yields
\begin{equation}\label{eq:theta-preest}
\theta_r^\alpha~\le~ \pr{\Rc_{\beta r}>\ell}+2\,\pr{B_{\alpha r}\leftrightarrow\partial B_{\beta r}}\pr{B_{\beta r+\ell}\leftrightarrow\partial B_r}.
\end{equation}
We now cover $\partial B_{\beta r+\ell}$ by balls $\tilde B$ of radius $\alpha (r-\beta r-\ell)$, for which we have
$$
\prmm{\tilde B \leftrightarrow \partial B_r} ~\le~ \prmm{B_{\alpha (r-\beta r-\ell)} \leftrightarrow \partial B_{r-\beta r-\ell}}=\theta^\alpha_{r-\beta r-\ell}.
$$
Since we can cover $\partial B_{\beta r+\ell}$ with at most $C(\frac{\beta r+\ell}{\alpha (r-\beta r-\ell)})^{d-1}$ such balls, this yields by a union bound,
\begin{equation}\label{eq:cover-arg}
\pr{B_{\beta r+\ell}\leftrightarrow\partial B_r}~\le~ C(\tfrac{\beta r+\ell}{\alpha (r-\beta r-\ell)})^{d-1}\,\theta^\alpha_{r-\beta r-\ell}.
\end{equation}
Similarly, covering $\partial B_{\alpha r}$ with at most $(\frac{1}{\beta-\alpha})^{d-1}$ balls of radius $\alpha r(\beta-\alpha)$, we obtain
\[\pr{B_{\alpha r}\leftrightarrow\partial B_{\beta r}}~\le~  C(\tfrac{1}{\beta-\alpha})^{d-1}\,\theta^\alpha_{r(\beta-\alpha)},\]
and the conclusion follows.
\end{proof}

Based on this renormalization inequality, we may now deduce estimates on connectivity probabilities by induction on scales.

\begin{proof}[Proof of Theorem~\ref{th:percol}]
Recall the notation $\theta_r^\alpha=\pr{B_{\alpha r}\leftrightarrow\partial B_r}$ for connectivity probabilities, and set for abbreviation $\pi_r^\alpha:=\prmm{\Rc_{r/2}>\alpha r}$.
We split the proof into three steps.

\medskip
\step1 Preliminary: for all $\alpha\in(0,1)$ there holds
\begin{equation}\label{eq:bnd-sr}
\theta^\alpha_s\le C(\tfrac sr)^{d-1}\theta^\alpha_r\qquad\text{for all $0<r< s(1-\alpha)$.}
\end{equation}
This follows from a similar covering argument as in~\eqref{eq:cover-arg}.
Indeed, given $\kappa\in(0,1)$, covering $\partial B_{\alpha s}$ with at most $C(\frac{1}{\kappa(1-\alpha)})^{d-1}$ balls of radius $\kappa\alpha s(1-\alpha)$, we find
\begin{eqnarray*}
\theta_s^\alpha\,=\,\pr{B_{\alpha s}\leftrightarrow\partial B_s}&\le&C(\tfrac{1}{\kappa(1-\alpha)})^{d-1}\,\prm{B_{\kappa\alpha s(1-\alpha)}\leftrightarrow\partial B_{s(1-\alpha)}}\\
&\le&C(\tfrac{1}{\kappa(1-\alpha)})^{d-1}\,\prm{B_{\kappa\alpha s(1-\alpha)}\leftrightarrow\partial B_{\kappa s(1-\alpha)}}\\
&=&C(\tfrac{1}{\kappa(1-\alpha)})^{d-1}\,\theta_{\kappa s(1-\alpha)}^{\alpha},
\end{eqnarray*}
and the claim~\eqref{eq:bnd-sr} follows.

\medskip
\step2 Proof that for all $\alpha\in(0,\frac12)$ there exists $\e_\alpha>0$ such that the following implication holds: for any $r_0>0$,
\begin{equation}\label{eq:impl-small}
\left.\begin{array}{ll}
\theta_{r_0}^\alpha\le\e_\alpha&\\
\pi_r^\alpha\le\frac12\e_\alpha&\text{for all $r\ge r_0$}
\end{array}\right\}
\quad\Longrightarrow\quad
\theta_r^\alpha\le C(1-2\alpha)^{1-d}\,\e_\alpha\quad\text{for all $r\ge2r_0$}.
\end{equation}
Let $\alpha\in(0,\frac12)$ be fixed, and set for abbreviation $\gamma:=\frac12-\alpha$. Applying Lemma~\ref{lem:renorm} with $\beta=\frac12$ and $\ell=\alpha r$, we find for all $r>0$, provided that $\pi_r^\alpha\le \frac12$,
\begin{equation}\label{eq:renorm-appl}
\theta_r^\alpha\,\le\, \pi_r^\alpha+ C(\alpha\gamma^2)^{1-d}\,(\theta_{\gamma r}^\alpha)^2.
\end{equation}
Choose $\e_\alpha>0$ small enough such that $C(\alpha \gamma^2)^{1-d}\,\e_\alpha\le\frac12$,
and assume that $\theta_{r_0}^\alpha\le\e_\alpha$ and that~$\pi_r^\alpha\le\frac12\e_\alpha$ for all $r\ge r_0$. 
Iterating the above inequality, we are led to
\begin{equation}\label{eq:preconcl-integer}
\theta_{r_0/\gamma^k}^\alpha\le\e_\alpha\qquad\text{for all integer $k\ge0$.}
\end{equation}
Combining this with~\eqref{eq:bnd-sr}, this yields the claim $\theta_r^\alpha\le C\gamma^{1-d}\e_\alpha$ for all $r\ge\frac1{1-\alpha}r_0$.

\medskip
\step3 Conclusion.\\
We only display the proof of~(ii) while the proof of~(i) is identical.
Let $\alpha\in(0,\frac12)$, recall the notation $\gamma=\frac12-\alpha$, and assume that $\pi_r^\alpha\le C_0\exp(-\frac1{C_0}r^{\kappa})$ for all $r$.
This decay assumption ensures that there exists $r_0$ such that $\pi_r^\alpha\le \frac12$ for all $r\ge r_0$.
The inequality~\eqref{eq:renorm-appl} then becomes for all $r\ge r_0$,
\[\theta_r^\alpha\,\le\, C_0\exp(-\tfrac1{C_0}r^{\kappa})+ C(\alpha\gamma^2)^{1-d}\,(\theta_{\gamma r}^\alpha)^2.\]
Provided $\kappa<\frac{\log2}{|\!\log\gamma|}$,
using~\eqref{eq:impl-small} as a starting point,
a direct induction argument allows to conclude for some $C_1$,
\[\theta_r^\alpha\le C_1\exp(-\tfrac1{C_1}r^{\kappa})\qquad\text{for all $r>0$}.\qedhere\]
\end{proof}

\subsection{Proof of Theorem~\ref{th:as-perc}}\label{sec:proof-mainpercol}
Let $0<\rho \le \frac12$ and $r_0\ge2$ to be fixed later, set $\rho':={\rho}/{(3\sqrt d)}$, and define a $\rho'\Z^d$-stationary random field $\eta$ on $\R^d$ by setting
\[\eta|_{Q_{\rho'}(z)}\,:=\,\left\{\begin{array}{lll}
1&:&\text{if $Q_{\rho'}(z)\cap I_n\ne\varnothing$ for some $n$ with $\displaystyle\inf_{m:m\ne n}\dist(I_n,I_m)<\rho$,}\\
0&:&\text{otherwise}.
\end{array}\right.\]
By definition of clusters $\{J_{p,\rho}\}_p$, we note that a cube of side length $\rho'$ can intersect at most one single cluster.
Covering $\partial B_{r_0/2}+B_\rho$ with $Cr_0^{d-1}$ unit balls, a union bound and the stationarity condition then yield
\begin{eqnarray}
\pr{B_{r_0/2}\leftrightarrow\partial B_{r_0}}
&\le&\pr{\exists p:J_{p,\rho}\cap(\partial B_{r_0/2}+B_\rho)\ne\varnothing,\,\diam J_{p,\rho}>\tfrac12r_0-1}\nonumber\\
&\lesssim&r_0^{d-1}\pr{\exists p:J_{p,\rho}\cap B\ne\varnothing,\,\diam J_{p,\rho}>\tfrac12r_0-1}.\label{eq:firstbnd-PBr2B2}
\end{eqnarray}
We split the proof into two steps, separately considering the two conditions of the alternative in~\ref{Hd''}.

\medskip
\step1 Conclusion under the smallness condition in~\ref{Hd''}.\\
Starting from~\eqref{eq:firstbnd-PBr2B2}, covering $B$ with $C\rho^{-d}$ cubes of side length $\rho'$, and noting that each such cube can intersect at most one single cluster, a union bound and the stationarity condition yield
\begin{eqnarray*}
\pr{B_{r_0/2}\leftrightarrow\partial B_{r_0}}
\,\lesssim\,\rho^{-d}\,r_0^{d-1}\pr{\diam J_{0,\rho}>\tfrac12r_0-1}.
\end{eqnarray*}
Therefore, for fixed $\e_\alpha>0$, if the smallness condition in~\ref{Hd''} holds with some $h_0,r_0>0$ and with $\frac1L\ll\e_\alpha$, then we deduce $\pr{B_{r_0/2}\leftrightarrow\partial B_{r_0}}<\e_\alpha$ provided that $\rho\le h_0$. In addition, in view of~\ref{HAR}, Lemma~\ref{lem:dep/act-rad} ensures the existence of dependence radii with $\pr{\Rc_r^z>\ell}\le C'(1+\frac r\ell)^{d-1}\ell^{-\kappa}$ for all $r,\ell$.
We can now apply Theorem~\ref{th:percol}(i) and the conclusion follows.

\medskip
\step2 Conclusion under the asymptotic smallness condition in~\ref{Hd''}.\\
Set $\delta':=\delta/(3\sqrt d)$. Starting from~\eqref{eq:firstbnd-PBr2B2}, noting that the interior ball condition ensures that each cluster contains at least one vertex point of $\delta'\Z^d$, a union bound and the stationarity condition yield
\begin{eqnarray*}
\pr{B_{r_0/2}\leftrightarrow\partial B_{r_0}}
\,\lesssim\,\delta^{-d}\,r_0^{d-1}\pr{\diam J_{0,\rho}>\tfrac12r_0-1}.
\end{eqnarray*}
Therefore, for fixed $\e_\alpha>0$, if the asymptotic smallness condition in~\ref{Hd''} holds with some~\mbox{$r_0>0$} and with $\frac1L\ll\e_\alpha$, then we can choose $\rho$ small enough such that we get $\pr{B_{r_0/2}\leftrightarrow\partial B_{r_0}}<\e_\alpha$, and the conclusion then follows as in Step~1.\qed

\begin{rem}\label{rem:no-coll-pr}
We prove the result stated in Remark~\ref{rem:no-coll}, claiming that the asymptotic smallness condition in~\ref{Hd''} holds in the stronger form~\eqref{eq:smallness+} whenever there is almost surely no contact between inclusions.
For that purpose, by a union bound and the stationarity condition, we find for all $r\ge1$ and~\mbox{$h\in(0,1)$},
\begin{eqnarray*}
{\pr{\diam J_{0,h}>r}}
&\le&\pr{\diam I_0>r}+\prm{\text{$\diam I_0\le r$ and $\exists m\ne 0:\dist(I_0,I_m)\le h$}}\\
&\lesssim&\pr{\diam I_0>r}+r^d\,\prm{\exists m\ne 0:\dist(I_0,I_m)\le h\text{ and }I_m\cap B\ne\varnothing}.
\end{eqnarray*}
For notational simplicity, we can assume that the enumeration of inclusions $\{I_n\}_n$ is such that almost surely $I_0$ is the inclusion that is the closest to the origin, $I_1$ is the second closest, etc., where we use the lexicographic order when needed to ensure uniqueness. The above can then be further estimated as
\begin{equation*}
\pr{\diam J_{0,h}>r}\,\lesssim\,\pr{\diam I_0>r}+r^d\,\expecM{\sum_{m=1}^\infty\mathds1_{\dist(I_0,I_m)\le h}\,\mathds1_{\sharp\{n:I_n\cap B\ne\varnothing\}> m}},
\end{equation*}
hence, by Hölder's inequality,
\begin{multline*}
\pr{\diam J_{0,h}>r}\\
\,\lesssim\,\pr{\diam I_0>r}
+r^d\,\sum_{m=1}^\infty\pr{\dist(I_0,I_m)\le h}^{1-\alpha}\pr{\sharp\{n:I_n\cap B\ne\varnothing\}> m}^\alpha.
\end{multline*}
The condition on the absence of contacts between inclusions yields \mbox{$\pr{\dist(I_0,I_m)\le h}\to0$} as $h\downarrow0$ for all $m\ne 0$.
In view of the local finiteness condition, Lebesgue's dominated convergence theorem then entails
\[\limsup_{h\downarrow0}\pr{\diam J_{0,h}>r}\,\lesssim\,\pr{\diam I_0>r},\]
and the conclusion~\eqref{eq:smallness+} follows from the control on inclusion diameters.
\end{rem}

\section{Effective viscosity problem}

This section is devoted to the proofs of Proposition~\ref{prop:cor} and Theorem~\ref{th:homog}  on the effective viscosity problem.

\subsection{Truncation result}
 Given $\rho>0$, recall that the sets $\{K_{p,\rho}=J_{p,\rho}+\rho B\}_p$ are all connected and disjoint by definition.
Yet, it might happen that for some $p$ the complement $\R^d\setminus K_{p,\rho}$ is not connected, which would in fact be highly problematic in the truncation arguments below. To circumvent this issue, we rather define the envelope $\tilde K_{p,\rho}$ of $K_{p,\rho}$ as the complement of the unbounded connected component of $\R^d\setminus K_{p,\rho}$ (this does not change the diameter). We then set
\[\tilde J_{p,\rho}\,:=\,\big\{x\in\tilde K_{p,\rho}:\dist(x,\partial \tilde K_{p,\rho})>\rho\big\},\]
and we define the fattened sets
\begin{eqnarray}\label{eq:defJ'J''}
J_{p,\rho}'&:=&\Big\{x\in \tilde J_{p,\rho}+2\rho B:\dist\big(x,\partial(\tilde J_{p,\rho}+2\rho B)\big)>\tfrac32\rho\Big\},\\
J_{p,\rho}''&:=&\Big\{x\in \tilde J_{p,\rho}+2\rho B:\dist\big(x,\partial(\tilde J_{p,\rho}+2\rho B\big)>\rho\Big\}.\nonumber
\end{eqnarray}
As by definition the sets $\{\tilde K_{p,\rho}=\tilde J_{p,\rho}+\rho B\}_p$ are all disjoint and satisfy an interior ball condition with radius~$\rho$, it is easily checked that the fattened sets $\{J_{p,\rho}''\}_p$ are disjoint as well.
For all $p$, we note that
\[J_{p,\rho}\subset\tilde{J}_{p,\rho}\subset J_{p,\rho}'\subset J_{p,\rho}'',\]
and the following properties are direct consequences of~\eqref{eq:defJ'J''}, using that \mbox{$\tilde J_{p,\rho}+\rho B$} has connected complement:
\begin{enumerate}[---]
\item
$J_{p,\rho}'$ is connected, has connected complement, and satisfies an interior ball condition with radius $\frac32\rho$ and an exterior ball condition with radius $\frac12\rho$;
\item
$J_{p,\rho}''$ is connected, has connected complement, and satisfies an interior ball condition with radius $\rho$ and an exterior ball condition with radius $\rho$.
\end{enumerate}
These properties ensure in particular that for all $p$ the set $J_{p,\rho}''\setminus J_{p,\rho}'$ is connected and satisfies an interior ball condition with radius $\frac12\rho$ (which entails, in particular, that it is a John domain, cf.~below).

With these appropriate definitions at hand for cluster neighborhoods, we can now state the following truncation result around clusters.

\begin{lem}[Truncation around clusters]\label{lem:trunc}
Let $\Ic$ satisfy~\ref{Hd} and~\ref{Hh}, and let $\rho>0$ be fixed.
For all $p$ there exists a truncation operator~$T_{p,\rho}$ such that for all divergence-free fields $g\in C^1_b(J_{p,\rho}'')^d$ the truncation \mbox{$T_{p,\rho}[g]\in H^1_0(J_{p,\rho}'')^d$} satisfies
\[\Div(T_{p,\rho}[g])=0,\qquad\D(T_{p,\rho}[g])|_{J_{p,\rho}'}=\D(g)|_{J_{p,\rho}'},\]
and
\[\|T_{p,\rho}[g]\|_{H^1(J_{p,\rho}'')}\,\lesssim_\rho\,(1+\diam J_{p,\rho})^C\,\|\!\D(g)\|_{\Ld^2(J_{p,\rho}'')}.\qedhere\]
\end{lem}

The proof follows from a standard construction based on the Bogovskii operator. Due to the possible geometric complexity of the clusters~$\{J_{p,\rho}\}_p$, we appeal to a suitably refined version of Bogovskii's construction due to Acosta, Dur\'{a}n, and Muschietti~\cite{ADM-06} in John domains, as stated below.
Recall that a set $V\subset\R^d$ is a \emph{John domain} if it has a basepoint~$x_V$ such that any other point can be connected to it without getting too close to the boundary: more precisely, $V$ is a John domain with John constant~$L_V$ with respect to a basepoint~$x_V$ if for all $y\in V$ there exists a Lipschitz map \mbox{$\zeta:[0,|y-x_V|]\to V$} with Lipschitz constant~$L_V$ such that $\zeta(0)=y$, $\zeta(|y-x_V|)=x_V$, and such that \mbox{$\dist(\zeta(t),\partial V)\ge\frac1{L_V}t$} for all $t\in[0,|y-x_V|]$.
Note that the dependence on~$L_V$ and $\diam V$ in the estimate below can be sharpened by inspecting the construction in~\cite{ADM-06}; this is however not pursued here and we only retain polynomial growth.

\begin{lem}[Bogovskii's construction in John domains; \cite{ADM-06}]\label{lem:Bogo}
Let $V\subset\R^d$ be a John domain with John constant~$L_V\ge1$, with respect to a basepoint $x_V$.
For all $1<q<\infty$ and $f\in\Ld^q(\Omega)$ with~$\int_V f=0$, there exists $u_f\in W^{1,q}_0(V)^d$ with
$\Div(u_f)=f$
and
\begin{align*}
\|\nabla u_f\|_{\Ld^{q}(V)}\,\lesssim_q\,\big(L_V+\diam V\big)^C\,\|f\|_{\Ld^q(V)},
\end{align*}
where the multiplicative constant may additionally depend on $\dist(x_V,\partial V)$.
\end{lem}

As a standard corollary of this Bogovskii construction, we deduce the following version of Korn's inequality in John domains; a short proof is included for convenience.

\begin{cor}[Korn's inequality in John domains]\label{cor:Korn}
Let \mbox{$V\subset\R^d$} be a John domain with John constant~$L_V\ge1$, with respect to a basepoint $x_V$.
For all $1<q<\infty$ and $u\in W^{1,q}(V)^d$ we have
\[\inf_{\Upsilon\in\R^d}\,\inf_{\Theta\in\Md^\Skew}\,\|u-\Upsilon-\Theta x\|_{W^{1,q}(V)}\,\lesssim_q\,\big(L_V+\diam V\big)^C\|\!\D(u)\|_{\Ld^q(V)},\]
where the multiplicative constant may additionally depend on $\dist(x_V,\partial V)$.
\end{cor}

\begin{proof}
For $1<q<\infty$, given $g\in \Ld^q(V)$ with $\int_V g=0$, letting $u_f$ be as in Lemma~\ref{lem:Bogo}, we can write by duality,
\begin{eqnarray*}
\|g\|_{\Ld^q(V)}&=&\sup\Big\{\int_V fg~:~f\in\Ld^{q'}(V),~\int_V f=0,~\|f\|_{\Ld^{q'}(V)}=1\Big\}\\
&=&\sup\Big\{\int_V u_f\cdot\nabla g~:~f\in\Ld^{q'}\!(V),~\int_V f=0,~\|f\|_{\Ld^{q'}(V)}=1\Big\},
\end{eqnarray*}
hence, using the bound of Lemma~\ref{lem:Bogo} on $\nabla u_f$, and appealing to Poincaré's inequality on the John domain $V$, cf.~\cite{Reshetnyak-80,Martio-88},
\begin{equation}
\|g\|_{\Ld^q(V)}\,\lesssim_q\,\big(L_V+\diam V\big)^C\|\nabla g\|_{W^{-1,q}(V)},\label{eq:dual-bogo}
\end{equation}
and similarly,
\begin{equation}\label{eq:dual-bogo-bis}
\|g\|_{\Ld^q(V)}
\,\lesssim_q\,\big(L_V+\diam V\big)^C\|\nabla g\|_{\Ld^q(V)}.
\end{equation}
Given $u\in W^{1,q}(V)^d$, choosing $x_0:=\fint_V x\,dx$, $\Upsilon_u:=\fint_V u$, and $\Theta_u:=\fint_V(\nabla u-\D(u))$, the inequality~\eqref{eq:dual-bogo-bis} yields
\[\|u-\Upsilon_u-\Theta_u(x-x_0)\|_{\Ld^q(V)}\,\lesssim_q\,\big(L_V+\diam V\big)^C\|\nabla u-\Theta_u\|_{\Ld^q(V)},\]
where the last factor can be split into
\begin{eqnarray*}
\|\nabla u-\Theta_u\|_{\Ld^q(V)}&\le&\|\!\D(u)\|_{\Ld^q(V)}+\Big\|(\nabla u-\D(u))-\fint_V(\nabla u-\D(u))\Big\|_{\Ld^q(V)}\\
&\le&\|\!\D(u)\|_{\Ld^q(V)}+\Big\|\nabla u-\fint_V\nabla u\Big\|_{\Ld^q(V)}.
\end{eqnarray*}
It remains to estimate the last term.
For that purpose, we appeal to the inequality~\eqref{eq:dual-bogo} in form of
\[\Big\|\nabla u-\fint_V\nabla u\Big\|_{\Ld^q(V)}\,\lesssim_q\,\big(L_V+\diam V\big)^C\|\nabla^2u\|_{W^{-1,q}(\Omega)}.\]
Recalling the standard observation
\[\nabla_{ij}^2u_k\,=\,\nabla_i\D(u)_{jk}+\nabla_j\D(u)_{ki}-\nabla_k\D(u)_{ij},\]
see e.g.~\cite[(3.16)]{Duvaut-Lions}, we get
\[\|\nabla^2u\|_{W^{-1,q}(\Omega)}\,\lesssim_q\,\|\!\D(u)\|_{\Ld^q(\Omega)},\]
and the conclusion follows.
\end{proof}

With these estimates at hand, we may now turn to the proof of Lemma~\ref{lem:trunc} for truncation around clusters.

\begin{proof}[Proof of Lemma~\ref{lem:trunc}]
Let a vector field $g\in C^1_b(J''_{p,\rho})^d$ be fixed with $\Div(g)=0$ in $J''_{p,\rho}$, and let $\Upsilon_g\in\R^d$ and $\Theta_g\in\Md^\Skew$ to be chosen later.
Let $\chi_{p,\rho}\in C^\infty_c(\R^d)$  be a cut-off function with
\[\chi_{p,\rho}|_{J_{p,\rho}'}=1,\qquad\chi_{p,\rho}|_{\R^d\setminus J_{p,\rho}''}=0,\qquad|\nabla\chi_{p,\rho}|\lesssim\rho^{-1}.\]
From the properties of $J_{p,\rho}'$ and $J_{p,\rho}''$, in particular from the fact that $J_{p,\rho}''\setminus J_{p,\rho}'$ is connected and satisfies an interior ball condition with radius $\frac12\rho$, we infer that $J_{p,\rho}''\setminus J_{p,\rho}'$ is a John domain with constant $\lesssim_{\rho}1+|J_{p,\rho}|\lesssim(1+\diam J_{p,\rho})^d$, and its basepoint can be chosen at distance~\mbox{$\gtrsim_\rho1$} from its boundary.
In view of the compatibility relation
\begin{equation*}
\int_{J_{p,\rho}''\setminus J_{p,\rho}'}\Div(\chi_{p,\rho} (g-\Upsilon_g-\Theta_gx))\,=\,-\int_{\partial J_{p,\rho}'}(g-\Upsilon_g-\Theta_gx)\cdot\nu
\,=\,-\int_{J_{p,\rho}'}\Div(g)\,=\,0,
\end{equation*}
where $\nu$ is the outward unit normal vector on $\partial J_{p,\rho}'$,
we may  appeal to Bogovskii's construction in form of Lemma~\ref{lem:Bogo}, which provides us some \mbox{$v_{p,\rho}\in H^1_0(J_{p,\rho}''\setminus J_{p,\rho}')^d$} such that
\[\Div(v_{p,\rho})=\Div(\chi_{p,\rho} (g-\Upsilon_g-\Theta_gx))\qquad\text{in $J_{p,\rho}''\setminus J_{p,\rho}'$},\]
and
\begin{eqnarray*}
\|\nabla v_{p,\rho}\|_{\Ld^2(J_{p,\rho}''\setminus J_{p,\rho}')}\,\lesssim_{\rho}\,(1+\diam J_{p,\rho})^C\|\Div(\chi_{p,\rho} (g-\Upsilon_g-\Theta_gx))\|_{\Ld^2(J_{p,\rho}''\setminus J_{p,\rho}')}.
\end{eqnarray*}
Next, we define
\[T_{p,\rho} [g]\,:=\,\chi_{p,\rho} (g-\Upsilon_g-\Theta_gx)-v_{p,\rho}~~\in H^1_0(J_{p,\rho}'')^d,\]
which satisfies
\[\Div(T_{p,\rho}[g])=0,\qquad\D(T_{p,\rho}[g])|_{J_{p,\rho}'}=\D(g)|_{J_{p,\rho}'},\]
and
\begin{eqnarray*}
\|\nabla T_{p,\rho}[g]\|_{\Ld^2(J_{p,\rho}'')}\,\lesssim_{\rho}\,(1+\diam J_{p,\rho})^C\|g-\Upsilon_g-\Theta_gx\|_{H^1(J_{p,\rho}'')}.
\end{eqnarray*}
Since $J_{p,\rho}''$ is also a John domain with constant $\lesssim_{\rho}(1+\diam J_{p,\rho})^d$ for some basepoint at distance $\gtrsim_\rho1$ from the boundary, we may now appeal to Korn's inequality in form of Corollary~\ref{cor:Korn}: for a suitable choice of constants $\Upsilon_g\in\R^d$ and $\Theta_g\in\Md^\Skew$, we find
\[\|g-\Upsilon_g-\Theta_gx\|_{H^1(J_{p,\rho}'')}\,\lesssim_{\rho}\,(1+\diam J_{p,\rho})^C\|\!\D(g)\|_{\Ld^2(J_{p,\rho}'')},\]
and the conclusion follows.
\end{proof}

\subsection{Construction of correctors: proof of Proposition~\ref{prop:cor}}\label{sec:cor}
Let $E\in\Md_0^\Sym$ be fixed.
The proof is split into two steps; as in~\cite{D-20}, it is deduced from~\cite{DG-21a} together with an approximation argument.

\medskip
\step1 Approximation with uniformly separated particles.\\
For $s>0$, we consider the chopped inclusions
\[I_n^s\,:=\,\big\{x\in I_n:\dist(x,\partial I_n)>2s\big\}+s B,\qquad\Ic^s:=\bigcup_nI_n^s,\]
which still satisfy the conditions~\ref{Hd} and~\ref{Hh},
as well as interior and exterior ball conditions with radius $s$,
and which have minimal interparticle distance $\inf_{n\ne m}\dist(I_n^s,I_m^s)\ge2s$. In this uniform context, we can apply~\cite[Proposition~2.1]{DG-21a}: there exists a unique corrector~$\psi_E^s$ that satisfies the different properties stated in Proposition~\ref{prop:cor} with $\Ic$ replaced by $\Ic^s$. In addition, we shall show that the following moment bound holds uniformly with respect to the parameter~\mbox{$s>0$},
\begin{eqnarray}
\expecm{|\nabla\psi_E^s|^2}\,\lesssim\,|E|^2.\label{eq:bnd-kap-psi}
\end{eqnarray}

\medskip\noindent
We turn to the proof of this estimate.
For that purpose, in terms of the truncation operators~$\{T_{p,\rho}\}_p$ constructed in Lemma~\ref{lem:trunc}, we first consider the following stationary random vector field,
\[\phi_E^\circ\,:=\,-\sum_pT_{p,\rho}[E(x-x_{p,\rho})],\qquad x_{p,\rho}:=\fint_{J_{p,\rho}''}x\,dx,\]
and we show that it satisfies
\begin{equation}\label{eq:prop-phi0}
(\D(\phi_E^\circ)+E)|_{\Ic}=0,\quad\Div(\phi_E^\circ)=0,\quad\expec{\D(\phi_E^\circ)}=0,\quad\expec{|\!\D(\phi_E^\circ)|^2}\lesssim|E|^2.
\end{equation}
The first two properties follow from the construction of $T_{p,\rho}$ with
\[\D(T_{p,\rho}[E(x-x_{p,\rho})])|_{J_{p,\rho}'}=E\qquad\text{and}\qquad \Div(T_{p,\rho}[E(x-x_{p,\rho})])=0.\]
We turn to the proof of the third property in~\eqref{eq:prop-phi0}.
Given $\chi\in C^\infty_c(B)$ with $\int_{\R^d}\chi=1$, we have by stationarity, for all $R\ge1$,
\[\expec{\nabla\phi_E^\circ}\,=\,\expecM{R^{-d}\int_{B_R}\chi(\tfrac\cdot R)\,\nabla\phi_E^\circ}\,=\,-\expecM{R^{-d-1}\int_{B_R}\phi_E^\circ\otimes\nabla\chi(\tfrac\cdot R)},\]
where the second identity follows by integration by parts.
Inserting the definition of $\phi_E^\circ$ and appealing to Lemma~\ref{lem:trunc}, we deduce
\begin{eqnarray}
|\expec{\nabla\phi_E^\circ}\!|&\lesssim&\expecM{R^{-d-1}\sum_{p:J''_{p,\rho}\cap B_R\ne\varnothing}\int_{J_{p,\rho}''}|\phi_E^\circ|}\nonumber\\
&\le&\expecM{R^{-d-1}\sum_{p:J_{p,\rho}''\cap B_R\ne\varnothing}|J''_{p,\rho}|^\frac12\|T_{p,\rho}(E(x-x_{p,\rho}))\|_{\Ld^2(J''_{p,\rho})}}\nonumber\\
&\lesssim_\rho&|E|\,\expecM{R^{-d-1}\sum_{p:J_{p,\rho}''\cap B_R\ne\varnothing}|J''_{p,\rho}|(\diam J''_{p,\rho})^C}.\label{eq:laslim-E0}
\end{eqnarray}
The condition~\ref{Hd'} with $\kappa$ large enough yields $\E[(\diam J''_{0,\rho})^C]<\infty$, and we note that the ergodic theorem gives almost surely
\begin{equation}\label{eq:bnd-sumdiam}
\limsup_{R\uparrow\infty}|B_R|^{-1}\sum_{p:J_{p,\rho}''\cap B_R\ne\varnothing}|J_{p,\rho}''|(\diam J''_{p,\rho})^C\,\le\,\expec{(\diam J''_{0,\rho})^C}.
\end{equation}
Letting $R\uparrow\infty$ in~\eqref{eq:laslim-E0} and using the above, the claim $\expec{\nabla\phi_E^\circ}=0$ follows.
We turn to the proof of the last property in~\eqref{eq:prop-phi0}:
noting that Lemma~\ref{lem:trunc} yields
\[\|\!\D(T_{p,\rho}(E(x-x_{p,\rho})))\|_{\Ld^2(J_{p,\rho}'')}\,\lesssim_\rho\,|E|(\diam J''_{p,\rho})^C,\]
we find by stationarity, provided that~\ref{Hd'} holds for $\kappa$ large enough,
\[\expecm{|\!\D(\phi_E^\circ)|^2}\,\lesssim_\rho\,|E|^2\,\expecm{(\diam J''_{0,\rho})^C}\,<\,\infty,\]
as claimed.

\medskip\noindent
Now that the different properties~\eqref{eq:prop-phi0} of the stationary random vector field $\phi_E^\circ$ have been established, we use them to prove~\eqref{eq:bnd-kap-psi}. Since $\Ic^s\subset\Ic$, we can use $\phi_E^\circ$ as a test function in the variational problem~\eqref{eq:cor-variat} defining~$\psi_E^s$ (with $\Ic$ replaced by $\Ic^s$), which leads us to
\[\expecm{|\!\D(\psi_E^s)+E|^2}\,\le\,\expecm{|\!\D(\phi_E^\circ)+E|^2}\,\lesssim\,|E|^2.\]
Integrating by parts and using the incompressibility and the sublinearity of $\psi_E^s$ at infinity, cf.~\eqref{eq:conv-cor}, we easily find
\begin{equation*}
\expec{|\nabla\psi_E^s|^2}\,=\,2\,\expec{|\!\D(\psi_E^s)|^2},
\end{equation*}
and the conclusion~\eqref{eq:bnd-kap-psi} follows.

\medskip
\step2 Conclusion.\\
In view of the uniform bound~\eqref{eq:bnd-kap-psi}, we may consider a weak limit point $\nabla\psi_E$ of $\{\nabla\psi_E^s\}_{s>0}$ in $\Ld^2(\Omega;\Ld^2_\loc(\R^d)^{d\times d})$ as $s\downarrow0$. It follows that $\nabla\psi_E$ is stationary with vanishing expectation and finite second moments, that it satisfies $\Div(\psi_E)=0$ and $(\D(\psi_E)+E)|_{\Ic}=0$, and that $\D(\psi_E)$ is the unique solution of the limiting variational problem~\eqref{eq:cor-variat}.
Finally, the sublinearity property~\eqref{eq:conv-cor} is a standard consequence of the ergodic theorem for a random field with centered stationary gradient having bounded second moments.
\qed

\subsection{Homogenization result: proof of Theorem~\ref{th:homog}}\label{sec:Gamma}
For all $\e>0$, we consider the unique minimizer $u_\e\in H^1_0(U)^d$ of the Stokes problem~\eqref{eq:mini-Stokes}, that is,
\[\inf\bigg\{F_\e(v)-\int_{U\setminus\Ic_\e(U)}f\cdot v~:~v\in H^1_0(U)^d\bigg\},\]
in terms of the energy functional
\[F_\e(v)\,:=\,\left\{\begin{array}{lll}
\int_{U}|\!\D(v)|^2&:&\text{if $\Div(v)=0$ and $\D(v)|_{\Ic_\e(U)}=0$},\\
\infty&:&\text{otherwise},
\end{array}\right.\]
and we also consider the unique minimizer $\bar u$ of the variational problem
\[\inf\bigg\{\bar F(v)-(1-\lambda)\int_Uf\cdot v~:~v\in H^1_0(U)^d\bigg\},\]
in terms of the homogenized energy functional
\[\bar F(v)\,:=\,\left\{\begin{array}{lll}
\int_{U}\D(v):\Bb\D(v)&:&\text{if $\Div(v)=0$},\\
\infty&:&\text{otherwise}.
\end{array}\right.\]
By standard properties of $\Gamma$-convergence, e.g.~\cite{DalMaso-93,Braides-06}, as $\mathds1_{U\setminus\Ic_\e(U)}\cvf1-\lambda$ \mbox{weakly-*} in~$\Ld^\infty(U)$, Theorem~\ref{th:homog} is a consequence of the almost sure $\Gamma$-convergence of energy functionals $F_\e$ to~$\bar F$ on~$H^1_0(U)^d$ as~$\e\downarrow0$. In other words, it suffices to establish the following two properties,
\begin{enumerate}[(A)]
\item \emph{$\Gamma$-liminf inequality}: For any $v\in H^1_0(U)^d$ and any sequence $(v_\e)_\e$ with $v_\e\cvf v$ in~$H^1_0(U)^d$, we have
\[\liminf_{\e\downarrow0}F_\e(v_\e)\,\ge\, \bar F(v).\]
\item \emph{$\Gamma$-limsup inequality}: For any $v\in H^1_0(U)^d$ there is a sequence $(v_\e)_\e$ with $v_\e\cvf v$ in~$H^1_0(U)^d$ such that
\[\limsup_{\e\downarrow0}F_\e(v_\e)\,\le\, \bar F(v).\]
\end{enumerate}
We split the proof into two steps.  While the $\Gamma$-liminf inequality easily follows from an approximation argument, the $\Gamma$-limsup inequality requires a particular care as naïve recovery sequences in form of $2$-scale expansions do not satisfy the incompressibility constraint. We emphasize that this variational approach is of fundamental utility in the present setting. Indeed, in contrast with other approaches to homogenization~\cite{DG-21a,D-20}, the construction of recovery sequences does not require any detailed description of energy minimizers in the direct proximity of contact points: correctors can be used there directly without precise modulation (more precisely, recovery sequences like~\eqref{eq:weps} below are simply chosen as {\it affine} transformations of correctors in each cluster neighborhood $J'_{p,\rho}$).

\medskip
\step1 Proof of $\Gamma$-liminf inequality~(A).\\
As in the proof of Proposition~\ref{prop:cor} above,
for $s>0$, we consider the choped inclusions
\[I_n^s\,:=\,\big\{x\in I_n:\dist(x,\partial I_n)>2s\big\}+s B,\qquad\Ic^s:=\bigcup_nI_n^s,\]
which still satisfy the conditions~\ref{Hd} and~\ref{Hh}, as well as interior and exterior ball conditions with radius $s$, and which have minimal interparticle distance $\inf_{n\ne m}\dist(I_n^s,I_m^s)\ge2s$.
We then define $\Ic^s_\e(U):=\bigcup_{n\in\Nc_\e(U)}\e I_n^s$ and
we consider the corresponding energies
\[F_\e^s(v)\,:=\,\left\{\begin{array}{lll}
\int_{U}|\!\D(v)|^2&:&\text{if $\Div(u)=0$ and $\D(v)|_{\Ic_\e^s(U)}=0$},\\
\infty&:&\text{otherwise}.
\end{array}\right.\]
For fixed $s>0$, we can apply the homogenization result of~\cite{DG-21a},
which, reformulated in variational terms, ensures that, for any sequence $(v_\e)_\e$ with $v_\e\cvf v$ in $H^1_0(U)^d$, we have
\[\liminf_{\e\downarrow0}F_\e^s(v_\e)\,\ge\,\bar F^s(v)\,:=\,\left\{\begin{array}{lll}
\int_{U}\D(v):\Bb^s\D(v)&:&\text{if $\Div(v)=0$},\\
\infty&:&\text{otherwise},
\end{array}\right.\]
where the effective viscosity $\Bb^s$ is given for all $E\in\Md_0^\Sym$ by
\[E:\Bb^s E\,:=\,\expec{|\!\D(\psi_E^s)+E|^2},\]
in terms of the corrector $\psi_E^s$ for the problem with chopped inclusions $\Ic^s$.
As by definition $\Ic_\e^s(U)\subset\Ic_\e(U)$, we find $F_\e^s\le F_\e$, so that the above entails
\[\liminf_{\e\downarrow0}F_\e(v_\e)\,\ge\,\liminf_{\e\downarrow0}F_\e^s(v_\e)\,\ge\,\bar F^s(v).\]
Finally, by the proof of Proposition~\ref{prop:cor}, we have $\D(\psi_E^s)\cvf\D(\psi_E)$ in $\Ld^2(\Omega;\Ld^2_\loc(\R^d)^{d\times d})$, and thus
\[\liminf_{s\downarrow0}E:\Bb^s E\,\ge\, E:\Bb E,\]
which entails $\liminf_{s\downarrow0}\bar F^s(v)\ge \bar F(v)$, and the $\Gamma$-liminf inequality~(A) follows.

\medskip
\step2 Proof of $\Gamma$-limsup inequality~(B).\\
This is the crucial step.
Let $v\in H^1_0(U)^d$ be fixed with $\Div(v)=0$. By a density argument, we may assume that $v$ belongs to $C^\infty_c(U)^d$ and is supported in
\[U^{3\theta}:=\{x\in U:\dist(x,\partial U)>3\theta\}\]
for some $\theta>0$.
A natural choice for the approximating sequence $v_\e$ would be the two-scale expansion of $v$,
\[v+\e\psi_E(\tfrac\cdot\e)\partial_E v,\]
in terms of the corrector $\psi_E$ constructed in Proposition~\ref{prop:cor},
where we implicitly sum over $E$ in some orthonormal basis of $\Md_0^\Sym$.
However, this two-scale expansion is not divergence-free and does not satisfy the required rigidity constraint on $\Ic_\e(U)$:
careful local surgery is needed to modify this naïve two-scale expansion.
We split the proof  into four further steps, where we gradually modify the naïve two-scale expansion
to make it satisfy the incompressibility and rigidity constraints without increasing its energy too much.

\medskip
\substep{2.1} Rigidity constraint: truncated two-scale expansion.\\
We construct an approximating sequence $(w_\e)_\e\subset \Ld^2(\Omega;H^1_0(U)^d)$ with the following properties:
\begin{enumerate}[(a)]
\item almost surely there exists $\e_0>0$ such that $w_\e$ is supported in $U^{2\theta}$ for all $0<\e<\e_0$;
\smallskip\item $\D(w_\e)|_{\e I_n}=0$ for all $n$;
\smallskip\item $\Div(w_\e)$ is supported in $U\setminus\bigcup_p \e J_{p,\rho}'$ and for all $q<\frac{2d}{d-2}$ it satisfies $\Div(w_\e)\to0$ almost surely in $\Ld^q(U)$ provided that~\ref{Hd'} holds for $\kappa$ large enough;
\smallskip\item $w_\e-(v+\e\psi_E(\tfrac\cdot\e)\partial_Ev)\to0$ almost surely in $H^1_0(U)^d$.
\end{enumerate}\smallskip
Recalling the definition of fattened clusters $J_{p,\rho}\subset J_{p,\rho}'\subset J_{p,\rho}''$, cf.~\eqref{eq:defJ'J''}, we choose a family~$(\chi_{p,\rho})_p\subset C^\infty_c(\R^d)$ of cut-off functions with
\[\chi_{p,\rho}|_{J_{p,\rho}'}=1,\qquad\chi_{p,\rho}|_{\R^d\setminus J_{p,\rho}''}=0,\qquad|\nabla\chi_{p,\rho}|\lesssim\rho^{-1},\]
we consider the truncation operators given for all $g\in C^\infty_c(U)$ by
\begin{eqnarray}
T_0^\e[g]&:=&\Big(1-\sum_p\chi_{p,\rho}(\tfrac\cdot\e)\Big)g+\sum_p\chi_{p,\rho}(\tfrac\cdot\e)\Big(\fint_{\e J_{p,\rho}''}g\Big),\label{eq:def-T01}\\
T_1^\e[g]&:=&\Big(1-\sum_p\chi_{p,\rho}(\tfrac\cdot\e)\Big)g+\sum_p\chi_{p,\rho}(\tfrac\cdot\e)\bigg(\Big(\fint_{\e J''_{p,\rho}}g\Big)+\Big(\fint_{\e J''_{p,\rho}}\nabla_jg\Big)(x-\e x_{p,\rho})_j\bigg),\nonumber
\end{eqnarray}
with $x_{p,\rho}:=\fint_{J_{p,\rho}''}x\,dx$, and we define the following modification of the naïve two-scale expansion,
\begin{equation}\label{eq:weps}
w_\e\,:=\,T_1^\e[v]+\e\psi_E(\tfrac\cdot\e)T_0^\e[\partial_Ev].
\end{equation}
It remains to check the properties~(a)--(d) and we start with~(a). Writing~$\e\lesssim_\rho|\e J_{p,\rho}''|^\frac{1}d$, we can estimate for all $\eta<1$,
\begin{eqnarray*}
\sup_{p:U\cap\e J_{p,\rho}''\ne\varnothing}\e\diam J_{p,\rho}''
&\lesssim&\e^\eta\sup_{p:U\cap\e J_{p,\rho}''\ne\varnothing}|\e J_{p,\rho}''|^\frac{1-\eta}d\diam J_{p,\rho}''\\
&\lesssim&\e^\eta\bigg(\sum_{p:U\cap\e J_{p,\rho}''\ne\varnothing}|\e J_{p,\rho}''|(\diam J_{p,\rho}'')^\frac d{1-\eta}\bigg)^\frac{1-\eta}d,
\end{eqnarray*}
hence, in view of~\eqref{eq:bnd-sumdiam}, provided that~\ref{Hd'} holds for some $\kappa>d$,
\begin{equation}\label{eq:supremum-est0}
\lim_{\e\downarrow0}\sup_{p:U\cap\e J_{p,\rho}''\ne\varnothing}\e\diam J_{p,\rho}''
\,=\,0.
\end{equation}
Therefore, as $u$ is supported in~$U^{3\theta}$, there exists almost surely some $\e_0>0$ such that for~\mbox{$0<\e<\e_0$} any fattened cluster $\e J''_{p,\rho}$ intersecting $U^{3\theta}$ satisfies $\e J''_{p,\rho}\subset U^{2\theta}$. By definition of $w_\e$ and of the truncation operators $T_0^\e,T_1^\e$, this proves~(a).

\medskip\noindent
We turn to~(b), that is, the rigidity of $w_\e$ in the inclusions. For all $n$, choosing~$p$ with $I_n\subset J_{p,\rho}$, the property $(\D(\psi_E)+E)|_{I_n}=0$ of the corrector leads to
\begin{eqnarray*}
\D(w_\e)|_{\e I_n}&=&\Big(\fint_{\e J''_{p,\rho}}\D(v)\Big)+\D(\psi_E)(\tfrac\cdot\e)|_{\e I_n}\Big(\fint_{\e J_{p,\rho}''}\partial_Ev\Big)\\
&=&\Big(\fint_{\e J''_{p,\rho}}\D(v)\Big)-E\Big(\fint_{\e J''_{p,\rho}}\partial_E v\Big)
\,=\,0.
\end{eqnarray*}
We continue with the proof of~(c), that is, the approximate incompressibility. By definition of $w_\e$ and of the truncation operators $T_0^\e,T_1^\e$, using that~$v$ and~$\psi_E$ are divergence-free, a direct computation yields
\begin{multline*}
\Div(w_\e)\,=\,\e\psi_E(\tfrac\cdot\e)\cdot\Big(1-\sum_p\chi_{p,\rho}(\tfrac\cdot\e)\Big)\nabla\partial_Ev
-\e\psi_E(\tfrac\cdot\e)\cdot\sum_p\nabla(\chi_{p,\rho}(\tfrac\cdot\e))\Big(\partial_Ev-\fint_{\e J_{p,\rho}''}\partial_Ev\Big)\\
-\sum_p\nabla(\chi_{p,\rho}(\tfrac\cdot\e))\cdot\bigg(v-\Big(\fint_{\e J''_{p,\rho}}v\Big)-\Big(\fint_{\e J''_{p,\rho}}\nabla_jv\Big)(x-\e x_{p,\rho})_j\bigg).
\end{multline*}
Properties of the cut-off functions $\{\chi_{p,\rho}\}_p$ ensure that $\Div(w_\e)$ is supported in $U\setminus \bigcup_p\e J'_{p,\rho}$. In addition, we find for all $q\ge2$,
\begin{multline*}
\|\Div(w_\e)\|_{\Ld^q(U)}^q\,\lesssim_{\rho,q}\,\|\nabla^2v\|_{\Ld^\infty(U)}^q\\
\times\bigg(\|\e\psi_E(\tfrac\cdot\e)\|_{\Ld^q(U)}^q
+\sum_{p:U\cap\e J_{p,\rho}''\ne\varnothing}(\diam J_{p,\rho}'')^q\int_{\e J_{p,\rho}''}|\e\psi_E(\tfrac\cdot\e)|^q\\
+\e^q\sum_{p:U\cap \e J_{p,\rho}''\ne\varnothing}|\e J''_{p,\rho}|(\diam J_{p,\rho}'')^{2q}\bigg),
\end{multline*}
hence, by Hölder's and Jensen's inequalities, for all $2\le q<s<\frac{2d}{d-2}$,
\begin{multline*}
\|\Div(w_\e)\|_{\Ld^q(U)}^q\,\lesssim_{\rho,q}\,\|\nabla^2v\|_{\Ld^\infty(U)}^q\big(\e^q+\|\e\psi_E(\tfrac\cdot\e)\|_{\Ld^s(U)}^q\big)\\
\times\bigg(1+\Big(\sum_{p:U\cap\e J_{p,\rho}''\ne\varnothing}|\e J_{p,\rho}''|(\diam J''_{p,\rho})^\frac{sq}{s-q}\Big)^\frac{s-q}s+\sum_{p:U\cap\e J_{p,\rho}''\ne\varnothing}|\e J''_{p,\rho}|(\diam J''_{p,\rho})^{2q}\bigg).
\end{multline*}
In view of~\eqref{eq:bnd-sumdiam}, if~\ref{Hd'} holds for $\kappa$ large enough, then the last right-hand side factor is uniformly bounded as $\e\downarrow0$.
Combining this with the sublinearity of $\psi_E$ at infinity, cf.~\eqref{eq:conv-cor}, we conclude $\Div(w_\e)\to0$ almost surely in $\Ld^q(U)$, that is, item~(c).

\medskip\noindent
It remains to prove~(d).
By definition of $w_\e$ and of the truncation operators~$T_0^\e$ and~$T_1^\e$, direct estimates yield as above for all $2\le q<\frac{2d}{d-2}$,
\begin{multline*}
\|\nabla w_\e-\nabla(v+\e\psi_E(\tfrac\cdot\e)\partial_Ev)\|_{\Ld^2(U)}^2
\,\lesssim_{\rho}\,\|\nabla^2v\|_{\Ld^\infty(U)}^2\big(\e^2+\|\e\psi_E(\tfrac\cdot\e)\|_{\Ld^q(U)}^2\big)\\
\times\bigg(1+\Big(\sum_{p:U\cap\e J_{p,\rho}''\ne\varnothing}|\e J_{p,\rho}''|(\diam J''_{p,\rho})^\frac{2q}{q-2}\Big)^\frac{q-2}q+\sum_{p:U\cap\e J_{p,\rho}''\ne\varnothing}|\e J_{p,\rho}''|(\diam J''_{p,\rho})^4\bigg)\\
+\|\nabla^2v\|_{\Ld^\infty(U)}^2\|\nabla\psi_E(\tfrac\cdot\e)\|_{\Ld^2(U)}^2\sup_{p:U\cap\e J_{p,\rho}''\ne\varnothing}(\e\diam J''_{p,\rho})^2.
\end{multline*}
In view of~\eqref{eq:conv-cor} and~\eqref{eq:bnd-sumdiam}, if~\ref{Hd'} holds for $\kappa$ large enough, the first right-hand side term tends to $0$ almost surely.
Since the ergodic theorem yields almost surely
\[\|\nabla\psi_E(\tfrac\cdot\e)\|_{\Ld^2(U)}^2\to|U|\expec{|\nabla\psi_E|^2},\]
the second right-hand side term also tends to $0$ almost surely in view of~\eqref{eq:supremum-est0}.
Combined with Poincaré's inequality, this proves item~(d).

\medskip

\substep{2.2}  Incompressiblity constraint: global modification.
\\
We construct an approximating sequence $(w'_\e)_\e\subset \Ld^2(\Omega;H^1_0(U)^d)$ with the following properties:
\begin{enumerate}[(a$'$)]
\item almost surely there exists $\e_0>0$ such that $w'_\e$ is supported in $U^\theta$ for all $0<\e<\e_0$;
\smallskip\item $\D(w'_\e)|_{\e I_n}=0$ for all $n$;
\smallskip\item $\Div(w'_\e)$ is supported in $U\cap\bigcup_p\e (J''_{p,\rho}\setminus J'_{p,\rho})$ with $\int_{\e(J_{p,\rho}''\setminus J_{p,\rho}')}\Div(w'_\e)=0$ for all $p$, and for all $q<\frac{2d}{d-2}$ it satisfies $\Div(w'_\e)\to0$ almost surely in~$\Ld^q(U)$ provided that~\ref{Hd'} holds for $\kappa$ large enough;
\smallskip\item $w'_\e-(v+\e\psi_E(\tfrac\cdot\e)\partial_Ev)\to0$ almost surely in $H^1_0(U)^d$.
\end{enumerate}\smallskip
As in~\cite[Theorem~III.3.1]{Galdi} (see in particular Lemma~\ref{lem:Bogo}),
 Bogovskii's construction yields a vector field $z_\e\in H^1_0(U^{2\theta})^d$ such that $\Div(z_\e)\,=\,\Div(w_\e)$ and for all~$1< q<\infty$,
\begin{equation}\label{eq:bnd-nabz}
\|\nabla z_\e\|_{\Ld^q(U)}\,\lesssim_q\,\|\Div(w_\e)\|_{\Ld^q(U)}.
\end{equation}
We then define
\[w'_\e:=w_\e-T_0^\e[z_\e],\]
and it remains to check the properties~(a$'$)--(d$'$).
Item~(a$'$) follows similarly as~(a) in the previous step.
Item~(b$'$) follows from~(b) and the definition of~$T_0^\e$.

\medskip\noindent
We turn to the proof of~(c).
As $\Div(z_\e)=\Div(w_\e)$, a direct computation yields
\[\Div(w'_\e)\,=\,\sum_p\chi_{p,\rho}(\tfrac\cdot\e)\,\Div(w_\e)+\sum_p\nabla(\chi_{p,\rho}(\tfrac\cdot\e))\cdot\Big(z_\e-\fint_{\e J''_{p,\rho}}z_\e\Big).\]
On the one hand, by (c) and the properties of the cut-off functions $\{\chi_{p,\rho}\}_p$, this implies that
$\Div(w'_\e)$ is supported in $U\cap\bigcup_p\e (J''_{p,\rho}\setminus J'_{p,\rho})$
and satisfies for all $p$,
\begin{multline*}
\int_{\e(J_{p,\rho}''\setminus J_{p,\rho}')}\Div(w'_\e)
~=~\int_{\e J_{p,\rho}''}\chi_{p,\rho}(\tfrac\cdot\e)\,\Div(w_\e)+\int_{\e J_{p,\rho}''}\nabla(\chi_{p,\rho}(\tfrac\cdot\e))\cdot\Big(z_\e-\fint_{\e J_{p,\rho}''}z_\e\Big)\\
\,=~\int_{\e J_{p,\rho}''}\chi_{p,\rho}(\tfrac\cdot\e)\,\Div(w_\e)-\int_{\e J_{p,\rho}''}\chi_{p,\rho}(\tfrac\cdot\e)\,\Div(z_\e)~=~0.
\end{multline*}
On the other hand, for all $2\le q\le s<\frac{2d}{d-2}$, by Hölder's inequality, we can estimate
\begin{multline*}
\|\Div(w'_\e)\|_{\Ld^q(U)}^q\,\lesssim_{\rho,q}\,\|\Div(w_\e)\|_{\Ld^q(U)}^q+\sum_p(\diam J''_{p,\rho})^q\|\nabla z_\e\|_{\Ld^q(\e J_{p,\rho}'')}^q\\
\,\lesssim\,\|\Div(w_\e)\|_{\Ld^q(U)}^q+\|\nabla z_\e\|_{\Ld^s(U)}^q\bigg(\sum_{p:U\cap\e J_{p,\rho}''\ne\varnothing}|\e J_{p,\rho}''|(\diam J''_{p,\rho})^\frac{sq}{s-q}\bigg)^\frac{s-q}s.
\end{multline*}
In view of~\eqref{eq:bnd-sumdiam}, if~\ref{Hd'} holds for $\kappa$ large enough, the last right-hand side factor remains uniformly bounded almost surely as $\e\downarrow0$.
Combined with~\eqref{eq:bnd-nabz} and~(c), this proves~(c$'$).

\medskip\noindent
It remains to establish~(d$'$). By definition of~$T_0^\e$, by Hölder's and Jensen's inequalities, we have for all~$2\le q<\frac{2d}{d-2}$,
\begin{eqnarray*}
\|\nabla T_0^\e[z_\e]\|_{\Ld^2(U)}^2&\lesssim_\rho&\|\nabla z_\e\|_{\Ld^2(U)}^2+\sum_{p} (\diam J''_{p,\rho})^2\|\nabla z_\e\|_{\Ld^2(\e J''_{p,\rho})}^2\\
&\lesssim&\|\nabla z_\e\|_{\Ld^q(U)}^2\bigg(1+\sum_{p:U\cap\e J_{p,\rho}''\ne\varnothing}|\e J''_{p,\rho}|(\diam J_{p,\rho})^\frac{2q}{q-2}\bigg)^\frac{q-2}q.
\end{eqnarray*}
Combined with~\eqref{eq:bnd-nabz} and~(c), this proves $w'_\e-w_\e\to0$ almost surely in $H^1_0(U)^d$, so that~(d$'$) follows from~(d).

\medskip
\substep{2.3} Incompressibility constraint: local modifications.
\\
We construct an approximating sequence $(v_\e)_\e\subset\Ld^2(\Omega; H^1(U)^d)$ with the following properties:
\begin{enumerate}[(a$''$)]
\item almost surely there exists $\e_0>0$ such that $v_\e$ is supported in $U$ for all $0<\e<\e_0$;
\smallskip\item $\D(v_\e)|_{\e I_n}=0$ for all $n$;
\smallskip\item $\Div(v_\e)=0$;
\smallskip\item $v_\e-(v+\e\psi_E(\tfrac\cdot\e)\partial_Ev)\to0$ almost surely in $H^1_0(U)^d$.
\end{enumerate}\smallskip
In view of the properties of $J_{p,\rho}'$ and $J_{p,\rho}''$, we recall that the set $J_{p,\rho}''\setminus J_{p,\rho}'$ is a John domain with constant $\lesssim_{\rho}1+|J_{p,\rho}|\lesssim_{\rho}(1+\diam J_{p,\rho})^d$, and its basepoint can be chosen at distance~$\gtrsim_{\rho}1$ from the boundary.
We may then appeal to Bogovskii's construction in form of Lemma~\ref{lem:Bogo}, which provides us some $z_{p,\rho}^\e\in H^1_0(J_{p,\rho}''\setminus J_{p,\rho}')^d$ such that
\[\Div(z_{p,\rho}^\e)\,=\,\Div(w_\e')(\e\cdot)\qquad\text{in $J''_{p,\rho}\setminus J'_{p,\rho}$},\]
and
\[\|\nabla z_{p,\rho}^\e\|_{\Ld^2(J_{p,\rho}''\setminus J_{p,\rho}')}\,\lesssim_{\rho}\,(1+\diam J_{p,\rho})^C\|\Div(w_\e')(\e\cdot)\|_{\Ld^2(J_{p,\rho}''\setminus J_{p,\rho}')}.\]
Now we define
\[v_\e:=w'_\e-\sum_p\e z_{p,\rho}^\e(\tfrac\cdot\e),\]
and it remains to check the properties~(a$''$)--(d$''$). Item~(a$''$) follows similarly as~(a) and~(a$'$) in the previous steps.
Item~(b$''$) follows from (b$'$) as $v_\e|_{\e J_{p,\rho}'}=w'_\e|_{\e J_{p,\rho}'}$. Item~(c$''$) follows by construction. Finally, we compute for all $2\le q<\frac{2d}{d-2}$,
\begin{eqnarray*}
\|\nabla v_\e-\nabla w'_\e\|_{\Ld^2(U)}^2&=&\sum_p\|(\nabla z_{p,\rho}^\e)(\tfrac\cdot\e)\|_{\Ld^2(\e (J_{p,\rho}''\setminus J_{p,\rho}'))}^2\\
&\lesssim_{\rho}&\sum_p(1+\diam J_{p,\rho})^C\|\Div(w'_\e)\|_{\Ld^2(\e J_{p,\rho}'')}^2\\
&\le&\|\Div(w'_\e)\|_{\Ld^q(U)}^2\bigg(\sum_{p:U\cap\e J_{p,\rho}\ne\varnothing}|\e J_{p,\rho}|(1+\diam J_{p,\rho})^{\frac{Cq}{q-2}}\bigg)^\frac{q-2}q,
\end{eqnarray*}
so that item~(d$''$) follows from~(c$'$) and (d$'$).

\medskip
\substep{2.4} Conclusion.\\
Properties (a$''$)--(c$''$) of the approximating sequence $(v_\e)_\e$ ensure that $v_\e$ belongs to $H^1_0(U)^d$ for all~$\e<\e_0$ and that
\[F_\e(v_\e)\,=\,\int_U|\!\D(v_\e)|^2.\]
In view of~(d$''$), we find almost surely
\[\lim_{\e\downarrow0}\Big(F_\e(v_\e)-\int_U|\!\D(v+\e\psi_E(\tfrac\cdot\e)\partial_Ev)|^2\Big)=0.\]
Expanding the gradient,
\[\D(v+\e\psi_E(\tfrac\cdot\e)\partial_Ev)\,=\,(\D(\psi_E)+E)(\tfrac\cdot\e)\partial_Ev+\e\psi_E(\tfrac\cdot\e)\otimes_s\nabla\partial_Ev,\]
where $\otimes_s$ stands for the symmetric tensor product,
and noting that the sublinearity property~\eqref{eq:conv-cor} of $\psi_E$ entails $\e\psi_E(\tfrac\cdot\e)\otimes_s\nabla\partial_Ev\to0$ almost surely in $\Ld^2(U)^{d\times d}$,
we deduce almost surely,
\[\lim_{\e\downarrow0}\Big(F_\e(v_\e)-\int_U(\partial_Ev)(\partial_{E'}v)\big((\D(\psi_E)+E):(\D(\psi_{E'})+E')\big)(\tfrac\cdot\e)\Big)=0.\]
As by definition $E:\Bb E'=\expec{(\D(\psi_E)+E):(\D(\psi_{E'})+E')}$, cf.~\eqref{eq:def-B}, the $\Gamma$-limsup inequality~(B) now follows from the ergodic theorem.
\qed

\section{Homogenization of linear elasticity with unbounded stiffness}\label{sec:stiff}
This section is devoted to the proofs of Proposition~\ref{prop:cor-ell} and Theorem~\ref{th:deg-ell}. As proofs are similar to those of Proposition~\ref{prop:cor} and Theorem~\ref{th:homog}, we skip most details and we only emphasize how proofs are simplified in the compressible setting.

\begin{proof}[Proof of Proposition~\ref{prop:cor-ell}]
We start with the compressible case. Let~\ref{Kd} and~\ref{Kd'} hold for some $\lambda,\rho>0$ and some $\kappa>2$, and let $E\in\Md^\Sym$ be fixed.
For all $s\ge\lambda$, we consider the inclusions $\Ic_s=\bigcup_nI_{n,s}$ where $|\Aa(x)|>s$, and we consider the truncated coefficient field
\begin{equation}\label{eq:def-As-trunc}
\Aa^s\,:=\,\Aa\mathds1_{\R^d\setminus\Ic_{s}}+\Id\mathds1_{\Ic_{s}}.
\end{equation}
The uniform ellipticity condition in~\ref{Kd} ensures that $\Aa^s\le\Aa$ in the sense of quadratic forms, and in addition $\Aa^s\uparrow\Aa$ almost everywhere as $s\uparrow\infty$.
For fixed $s\ge\lambda$, as the truncated coefficient field $\Aa^s$ is stationary, ergodic, uniformly elliptic, and uniformly bounded, there exists a unique corrector field $\varphi^{1;s}_E$ that satisfies the different properties of Proposition~\ref{prop:cor-ell} with $\Ic$ replaced by $\Ic^s$.
For all $p$, let $\chi_{p,\lambda,\rho}\in C^\infty_c(\R^d)$ be a cut-off function around cluster~$J_{p,\lambda,\rho}$ with
\[\chi_{p,\lambda,\rho}|_{J_{p,\lambda,\rho}}=1,\qquad\chi_{p,\lambda,\rho}|_{\R^d\setminus(J_{p,\lambda,\rho}+\rho B)}=0,\qquad|\nabla\chi_{p,\lambda,\rho}|\lesssim\rho^{-1},\]
and consider the following stationary random vector field,
\begin{equation}\label{eq:def-phiEcirc}
\phi_{E}^\circ\,:=\,-\sum_p \chi_{p,\lambda,\rho}E(x-x_{p,\lambda,\rho}),\qquad x_{p,\lambda,\rho}:=\fint_{J_{p,\lambda,\rho}+\rho B}x\,dx.
\end{equation}
As~\ref{Kd'} holds for some $\kappa>2$, we deduce $\E[(\diam J_{0,\lambda,\rho})^2]<\infty$, and it is then easily checked that
\[(\D(\phi_E^\circ)+E)|_{\Ic_{\lambda}}=0,\qquad\expec{\nabla\phi_E^\circ}=0,\qquad\expec{|\nabla\phi_E^\circ|^2}\lesssim|E|^2.\]
Therefore, for all $s\ge\lambda$, using $\phi_E^\circ$ as a test function in the variational problem~\eqref{eq:corr-phi1} defining $\varphi_E^{1;s}$ (with $\Ic$ replaced by $\Ic^s$), and using the uniform ellipticity condition in~\ref{Kd}, we deduce
\begin{eqnarray*}
\expecm{|\!\D(\varphi^{1;s}_E)+E|^2}&\le&\expecm{(\D(\varphi^{1;s}_E)+E):\Aa^s(\D(\varphi^{1;s}_E)+E)}\\
&\le&\expecm{(\D(\phi_E^\circ)+E):\Aa^s(\D(\phi_E^\circ)+E)}\\
&\le&\lambda\expecm{|\!\D(\phi_E^\circ)+E|^2}\\
&\lesssim&\lambda|E|^2.
\end{eqnarray*}
Integrating by parts and using the sublinearity of $\varphi_E^{1;s}$ at infinity, cf.~\eqref{eq:conv-cor}, we note that
\begin{equation*}
\expecm{|\nabla\varphi_E^{1;s}|^2}\,\le\,\expecm{|\nabla\varphi_E^{1;s}|^2}+\expecm{|\Div(\varphi_E^{1;s})|^2}\,=\,2\,\expecm{|\!\D(\varphi_E^{1;s})|^2},
\end{equation*}
so the above ensures that the full gradient $\nabla\varphi_E^{1;s}$ is uniformly bounded in $\Ld^2(\Omega;\Ld^2_\loc(\R^d)^{d\times d})$ as $s\uparrow\infty$. Finally, we check that weak limit points coincide with the unique minimizer~$\D(\varphi^{1}_E)$ of the limiting variational problem~\eqref{eq:corr-phi1}.

\medskip\noindent
We turn to the incompressible case. Given $E \in \Md_0^\Sym$, the test function $\phi_E^\circ$ in~\eqref{eq:def-phiEcirc} must then be suitably modified to be divergence-free, for which we can proceed as in the proof of Proposition~\ref{prop:cor} based on the key truncation result in Lemma~\ref{lem:trunc}.
\end{proof}

\begin{proof}[Proof of Theorem~\ref{th:deg-ell}]
We start with the compressible case.
Let~\ref{Kd} and~\ref{Kd'} hold for some $\lambda,\rho>0$ and some $\kappa>d$.
As for Theorem~\ref{th:homog}, we use a $\Gamma$-convergence approach. For all $\e>0$, we consider the unique solution $u_\e^1$ of the variational problem~\eqref{eq:ellPDE}, that is,
\[\inf\bigg\{\tfrac12 F^1_\e(v)-\int_Uf\cdot v~:~v\in H^1_0(U)^d\bigg\},\]
in terms of the energy functional
\[F^1_\e(v)\,:=\,\int_U\D( v):\Aa(\tfrac\cdot\e)\D( v),\]
and we also consider the unique minimizer $\bar u$ of the variational problem
\[\inf\bigg\{\tfrac12\bar F^1(v)-\int_Uf\cdot v~:~v\in H^1_0(U)^d\bigg\},\]
in terms of the homogenized energy functional
\[\bar F^1(v)\,:=\,\int_U\D( v):\bar\Aa^1\D( v).\]
It suffices to prove the almost sure $\Gamma$-convergence of energy functionals $F_\e^1$ to $\bar F^1$ on $H^1_0(U)^d$ as $\e\downarrow0$. As in the proof of Theorem~\ref{th:homog}, the $\Gamma$-liminf inequality easily follows from the known result for the truncated problem with $\Aa$ replaced by $\Aa^s$, cf.~\eqref{eq:def-As-trunc}, together with the weak convergence $\nabla\varphi_E^{1;s}\cvf{}\nabla\varphi_E^1$ in $\Ld^2(\Omega;\Ld^2_\loc(\R^d)^{d\times d})$ as $s\uparrow\infty$.
It only remains to prove the $\Gamma$-limsup inequality: given $v\in C^\infty_c(U)^d$, we need to construct a sequence $(v^1_\e)_\e$ with $v^1_\e\cvf{}v$ in~$H^1_0(U)^d$ such that
\begin{equation}\label{eq:gamma-limsup-F1}
\limsup_{\e\downarrow}F_\e^1(v^1_\e)\,\le\,\bar F^1(v).
\end{equation}
For that purpose, we consider the truncated two-scale expansion
\begin{equation}\label{eq:def-veps1}
v^1_\e\,:=\,T_1^\e[v]+\e\varphi_E^1(\tfrac\cdot\e)T_0^\e[\partial_Ev],
\end{equation}
where we implicitly sum over $E$ in an orthonormal basis of $\Md^\Sym$,
in terms of the following truncation operators as in~\eqref{eq:def-T01},
\begin{eqnarray*}
T_0^\e[g]&:=&\Big(1-\sum_p\chi_{p,\lambda,\rho}(\tfrac\cdot\e)\Big)g+\sum_p\chi_{p,\lambda,\rho}(\tfrac\cdot\e)\Big(\fint_{\e (J_{p,\lambda,\rho}+\rho B)}g\Big),\\
T_1^\e[g]&:=&\Big(1-\sum_p\chi_{p,\lambda,\rho}(\tfrac\cdot\e)\Big)g\\
&&\qquad+\sum_p\chi_{p,\lambda,\rho}(\tfrac\cdot\e)\bigg(\Big(\fint_{\e (J_{p,\lambda,\rho}+\rho B)}g\Big)+\Big(\fint_{\e(J_{p,\lambda,\rho}+\rho B)}\nabla_jg\Big)(x-\e x_{p,\lambda,\rho})_j\bigg).
\end{eqnarray*}
Expanding the gradients, we find
\begin{multline*}
F_\e^1(v^1_\e)\,=\,\int_UT_0^\e[\partial_Ev]T_0^\e[\partial_{E'}v]\big((\D(\varphi^1_E)+E):\Aa(\D(\varphi^1_{E'})+E')\big)(\tfrac\cdot\e)\\
+2\int_UT_0^\e[\partial_Ev](\D(\varphi^1_E)+E)(\tfrac\cdot\e):\Aa(\tfrac\cdot\e)\Big(\D(T_1^\e[v])-T_0^\e[\D(v)]+\e\varphi^1_{E'}(\tfrac\cdot\e)\otimes_s\nabla T_0^\e[\partial_{E'}v]\Big)\\
+\int_U\Big(\D(T_1^\e[v])-T_0^\e[\D(v)]+\e\varphi^1_E(\tfrac\cdot\e)\otimes_s\nabla T_0^\e[\partial_Ev]\Big)\\
:\Aa(\tfrac\cdot\e)\Big(\D(T_1^\e[v])-T_0^\e[\D(v)]+\e\varphi^1_{E'}(\tfrac\cdot\e)\otimes_s\nabla T_0^\e[\partial_{E'}v]\Big),
\end{multline*}
and thus, noting that $\nabla T_1^\e[v]-T_0^\e[\nabla v]$ and $\nabla T_0^\e[\nabla v]$ vanish in $\e\Ic_\lambda$,
\begin{multline*}
\Big|F_\e^1(v^1_\e)\,-\,\int_UT_0^\e[\partial_Ev]T_0^\e[\partial_{E'}v]\big((\D(\varphi^1_E)+E):\Aa(\D(\varphi^1_{E'})+E')\big)(\tfrac\cdot\e)\Big|\\
\,\lesssim\,\lambda\int_U|T_0^\e[\partial_Ev]||(\nabla\varphi^1_E+E)(\tfrac\cdot\e)| \Big|\nabla T_1^\e[v]-T_0^\e[\nabla v]+\e\varphi^1_{E'}(\tfrac\cdot\e)\otimes\nabla T_0^\e[\partial_{E'}v]\Big|\\
+\lambda\int_U\Big|\nabla T_1^\e[v]-T_0^\e[\nabla v]+\e\varphi^1_E(\tfrac\cdot\e)\otimes\nabla T_0^\e[\partial_Ev]\Big|^2.
\end{multline*}
By definition of truncation operators, a direct computation leads us to
\begin{multline}\label{eq:bnd-Xeps}
\Big|F_\e^1(v^1_\e)\,-\,\int_UT_0^\e[\partial_Ev]T_0^\e[\partial_{E'}v]\big((\D(\varphi^1_E)+E):\Aa(\D(\varphi^1_{E'})+E')\big)(\tfrac\cdot\e)\Big|\\
\,\lesssim_{\rho}\,\lambda\|\nabla v\|_{W^{1,\infty}(U)}^2\,X_\e\Big(X_\e+\|(\nabla\varphi_E^1+E)(\tfrac\cdot\e)\|_{\Ld^2(U)}\Big),
\end{multline}
where we have set for abbreviation
\begin{multline*}
X_\e\,:=\,\|\e\varphi^1_E(\tfrac\cdot\e)\|_{\Ld^2(U)}
+\bigg(\e^2\sum_{p:U\cap\e J_{p,\lambda,\rho}^+\ne\varnothing}|\e J_{p,\lambda,\rho}^+|(\diam J_{p,\lambda,\rho}^+)^4\bigg)^\frac12\\
+\bigg(\sum_{p:U\cap\e J_{p,\lambda,\rho}^+\ne\varnothing}\int_{J_{p,\lambda,\rho}^+}|\e\varphi^1_E(\tfrac\cdot\e)|^2(\diam J_{p,\lambda,\rho}^+)^2\bigg)^\frac12,
\end{multline*}
and $J_{p,\lambda,\rho}^+:=J_{p,\lambda,\rho}+\rho B$.
For all $r>2$ and $0<\eta<2$, writing $\e\lesssim_\rho|\e J_{p,\lambda,\rho}^+|^{\frac{1}d}$, using the discrete $\ell^\frac{d+2-\eta}d$--$\ell^1$ inequality and Hölder's inequality, we find
\begin{multline*}
X_\e\,\lesssim_\rho\,\|\e\varphi^1_E(\tfrac\cdot\e)\|_{\Ld^2(U)}
+\e^\frac\eta2\bigg(\sum_{p:U\cap\e J_{p,\lambda,\rho}\ne\varnothing}|\e J_{p,\lambda,\rho}^+|(\diam J_{p,\lambda,\rho}^+)^\frac{4d}{d+2-\eta}\bigg)^\frac{d+2-\eta}{2d}\\
+\|\e\varphi^1_{E'}(\tfrac\cdot\e)\|_{\Ld^r(U)}\bigg(\sum_{p:U\cap\e J_{p,\lambda,\rho}\ne\varnothing}|\e J_{p,\lambda,\rho}|(\diam J_{p,\lambda,\rho})^\frac{2r}{r-2}\bigg)^\frac{r-2}{2r}.
\end{multline*}
Given $\kappa>d$, we can find $2<r<\frac{2d}{d-2}$ and $0<\eta<2$ such that $\frac{2r}{r-2}<\kappa$ and $\frac{4d}{d+2-\eta}<\kappa$. In view of condition~\ref{Kd'}, together with the sublinearity of $\varphi_E^1$ in form of $\e\varphi_E^1(\tfrac\cdot\e)\to0$ almost surely in $\Ld^q_\loc(\R^d)$ for all $q<\frac{2d}{d-2}$, we deduce $X_\e\to0$ almost surely.
Combined with~\eqref{eq:bnd-Xeps}, this yields almost surely,
\[\limsup_{\e\downarrow0}F_\e(v_\e^1)\,\le\,\limsup_{\e\downarrow0}\int_UT_0^\e[\partial_Eu]T_0^\e[\partial_{E'}u]\big((\D(\varphi^1_E)+E):\Aa(\D(\varphi^1_{E'})+E')\big)(\tfrac\cdot\e).\]
Since by Proposition~\ref{prop:cor-ell} the expression $(\D(\varphi^1_E)+E):\Aa(\D(\varphi^1_{E'})+E')$ is stationary and belongs to $\Ld^1(\Omega;\Ld^1_\loc(\R^d))$, and since $T_0^\e[\partial_Eu]T_0^\e[\partial_{E'}u]\to(\partial_Eu)(\partial_{E'}u)$ in $\Ld^\infty(U)$, the ergodic theorem yields the conclusion~\eqref{eq:gamma-limsup-F1}.

\medskip\noindent
In the incompressible case, the recovery sequence $v_\e^1$ in~\eqref{eq:def-veps1} must be suitably modified to be divergence-free, for which we can proceed as in the proof of Theorem~\ref{th:homog} based on local surgery using the key truncation result in Lemma~\ref{lem:trunc}.
\end{proof}

\section*{Acknowledgements}
The authors thank Hugo Duminil-Copin for his feedback on the subcritical percolation estimates in Section~\ref{sec:percol}.
MD acknowledges financial support from the CNRS-Momentum program,
and AG from the European Research Council (ERC) under the European Union's Horizon 2020 research and innovation programme (Grant Agreement n$^\circ$~864066).

\bibliographystyle{plain}

\begin{thebibliography}{10}

\bibitem{ADM-06}
G.~Acosta, R.~G. Dur\'{a}n, and M.~A. Muschietti.
\newblock Solutions of the divergence operator on {J}ohn domains.
\newblock {\em Adv. Math.}, 206(2):373--401, 2006.

\bibitem{BG-72b}
G.~K. Batchelor and J.T. Green.
\newblock The determination of the bulk stress in suspension of spherical
  particles to order $c^2$.
\newblock {\em J. Fluid Mech.}, 56(3):401--427, 1972.

\bibitem{BG-72a}
G.~K. Batchelor and J.T. Green.
\newblock The hydrodynamic interaction of two small freely-moving spheres in a
  linear flow field.
\newblock {\em J. Fluid Mech.}, 56(2):375--400, 1972.

\bibitem{Biskup-11}
M.~Biskup.
\newblock Recent progress on the random conductance model.
\newblock {\em Probab. Surv.}, 8:294--373, 2011.

\bibitem{Braides-06}
A.~Braides.
\newblock A handbook of ${\Gamma}$-convergence.
\newblock In {\em Handbook of differential equations: stationary partial
  differential equations. {V}ol. {3}}, Handb. Differ. Equ.
  Elsevier/North-Holland, Amsterdam, 2006.

\bibitem{DalMaso-93}
G.~Dal~Maso.
\newblock {\em An introduction to {$\Gamma$}-convergence}.
\newblock Progress in Nonlinear Differential Equations and their Applications,
  8. Birkh\"auser Boston Inc., Boston, MA, 1993.

\bibitem{D-20}
M.~Duerinckx.
\newblock Effective viscosity of random suspensions without uniform separation.
\newblock Preprint, arXiv:2008.13188.

\bibitem{DG-21b}
M.~Duerinckx and A.~Gloria.
\newblock On {E}instein's effective viscosity formula.
\newblock Preprint, arXiv:2008.03837.

\bibitem{DG20b}
M.~Duerinckx and A.~Gloria.
\newblock {Multiscale functional inequalities in probability: Constructive
  approach}.
\newblock {\em Ann. H. Lebesgue}, 3:825--872, 2020.

\bibitem{DG-21a}
M.~Duerinckx and A.~Gloria.
\newblock Corrector equations in fluid mechanics: {E}ffective viscosity of
  colloidal suspensions.
\newblock {\em Arch. Ration. Mech. Anal.}, 239:1025--1060, 2021.

\bibitem{DCRT-19}
H.~Duminil-Copin, A.~Raoufi, and V.~Tassion.
\newblock Exponential decay of connection probabilities for subcritical
  {V}oronoi percolation in {$\Bbb{R}^d$}.
\newblock {\em Probab. Theory Related Fields}, 173(1-2):479--490, 2019.

\bibitem{DCRT-20}
H.~Duminil-Copin, A.~Raoufi, and V.~Tassion.
\newblock Subcritical phase of {$d$}-dimensional {P}oisson-{B}oolean
  percolation and its vacant set.
\newblock {\em Ann. H. Lebesgue}, 3:677--700, 2020.

\bibitem{Duvaut-Lions}
G.~Duvaut and J.-L. Lions.
\newblock {\em Inequalities in mechanics and physics}, volume 219 of {\em
  Grundlehren der Mathematischen Wissenschaften}.
\newblock Springer-Verlag, Berlin-New York, 1976.

\bibitem{Einstein-05}
A.~Einstein.
\newblock {\"U}ber die von der molekularkinetischen {T}heorie der {W}\"arme
  geforderte {B}ewegung von in ruhenden {F}l\"ussigkeiten suspendierten
  {T}eilchen.
\newblock {\em Ann. Phys.}, 322(8):549--560, 1905.

\bibitem{Galdi}
G.~P. Galdi.
\newblock {\em An introduction to the mathematical theory of the
  {N}avier-{S}tokes equations. Steady-state problems}.
\newblock Springer Monographs in Mathematics. Springer, New York, second
  edition, 2011.

\bibitem{GV-GL-21}
D.~G\'erard-Varet and A.~Girodroux-Lavigne.
\newblock Homogenization of stiff inclusions through network approximation.
\newblock Preprint, arXiv:2106.06299.

\bibitem{Grimmett}
G.~Grimmett.
\newblock {\em Percolation}.
\newblock Springer-Verlag, New York, 1989.

\bibitem{Hofer-GV-20}
R.~Höfer and D.~Gérard-Varet.
\newblock {Mild assumptions for the derivation of Einstein's effective
  viscosity formula}.
\newblock Preprint, arXiv:2002.04846.

\bibitem{JKO94}
V.~V. Jikov, S.~M. Kozlov, and O.~A. Ole{\u\i}nik.
\newblock {\em Homogenization of differential operators and integral
  functionals}.
\newblock Springer-Verlag, Berlin, 1994.

\bibitem{LSS-97}
T.~M. Liggett, R.~H. Schonmann, and A.~M. Stacey.
\newblock Domination by product measures.
\newblock {\em Ann. Probab.}, 25(1):71--95, 1997.

\bibitem{Martio-88}
O.~Martio.
\newblock John domains, bi-{L}ipschitz balls and {P}oincar\'{e} inequality.
\newblock {\em Rev. Roumaine Math. Pures Appl.}, 33(1-2):107--112, 1988.

\bibitem{Penrose-01}
M.~D. Penrose.
\newblock Random parking, sequential adsorption, and the jamming limit.
\newblock {\em Comm. Math. Phys.}, 218(1):153--176, 2001.

\bibitem{Reshetnyak-80}
Yu.~G. Reshetnyak.
\newblock Integral representations of differentiable functions in domains with
  a nonsmooth boundary.
\newblock {\em Sibirsk. Mat. Zh.}, 21(6):108--116, 221, 1980.

\end{thebibliography}

\def\cprime{$'$} \def\cprime{$'$} \def\cprime{$'$}

\end{document}